\RequirePackage{fix-cm}
\documentclass[12pt]{amsart}
\usepackage{amsmath}
\usepackage{setspace}
\usepackage{fancyhdr}
\usepackage{mathrsfs}
\usepackage{xifthen}

\usepackage{geometry}
\usepackage{fancyhdr}
\usepackage{color}
\usepackage{textcomp}
\usepackage{amsthm}
\usepackage{amstext}
\usepackage{amsfonts}
\usepackage{amssymb}
\usepackage{fixltx2e}
\usepackage{graphicx}

\usepackage{enumitem}

\usepackage{tikz}
\usepackage{tikz-cd}
\usepackage{caption}
\usetikzlibrary{calc,backgrounds}
\usetikzlibrary{matrix,arrows,decorations.pathmorphing}

\usepackage{bigstrut}

\newcommand{\Spf}{\text{Spf}}
\newcommand{\Hom}{\text{Hom}}
\newcommand{\End}{\text{End}}

\newcommand{\Lie}{\text{Lie}}
\newcommand{\GL}{\text{GL}}

\newcommand{\Sh}{\text{Sh}}
\newcommand{\Ext}{\text{Ext}}

\newcommand{\Fil}{\text{Fil}}

\newcommand{\Ker}{\text{Ker}}

\newcommand{\Res}{\text{Res}}

\newcommand{\Z}{\mathbb{Z}}
\newcommand{\Q}{\mathbb{Q}}
\newcommand{\C}{\mathbb{C}}

\newcommand{\Gm}{\mathbb{G}_m}

\newcommand{\R}{\mathbb{R}}
\newcommand{\Af}{\mathbb{A}_{f}}
\newcommand{\Afp}{\mathbb{A}_{f}^{p}}
\newcommand{\inv}{^{-1}}

\newcommand{\hooklongrightarrow}{\lhook\joinrel\longrightarrow}

\newcommand{\D}{\mathbb{D}}
\newcommand{\module}{\Lambda}
\newcommand{\genring}{R}
\newcommand{\genscheme}{S}
\newcommand{\dieudonnemodule}{M}

\newcommand{\completion}{\widehat}
\newcommand{\genqisog}{\iota}
\newcommand{\borel}{B}
\newcommand{\levi}{L}
\newcommand{\parabolic}{P}
\newcommand{\unipotent}{U}

\newcommand{\witt}{W}
\newcommand{\quotientwitt}{K_0}
\newcommand{\BT}{X}
\newcommand{\deform}{\mathscr}

\newcommand{\nilpconnection}{\nabla}
\newcommand{\liftlow}{\textbf}
\newcommand{\DeligneLusztig}[3]{X_{\{#2\}}^{#1}([#3])}
\newcommand{\Def}{\text{Def}}

\newcommand{\weylgroup}{\Omega}
\newcommand{\characters}{X^*}
\newcommand{\cocharacters}{X_*}
\newcommand{\roots}{\Phi}
\newcommand{\coroots}{\Phi^\vee}

\newcommand{\newtonmap}{\nu}
\newcommand{\kottwitzmap}{\kappa}
\newcommand{\newtonset}{\mathcal{N}}
\newcommand{\unramext}{E}
\newcommand{\integerring}{\mathscr{O}}
\newcommand{\unram}{\text{un}}
\newcommand{\EL}{\widetilde}
\newcommand{\ELgroup}[1][]{%
\ifthenelse{\isempty{#1}}{{\text{Res}_{\unramext|\mathbb{Q}_p} \GL_n}}{\text{Res}_{\unramext_{#1}|\mathbb{Q}_p} \GL_{n_{#1}}}%
}
\newcommand{\ELgroupintegral}[1][]{%
\ifthenelse{\isempty{#1}}{{\text{Res}_{\integerring|\mathbb{Z}_p} \GL_n}}{\text{Res}_{\integerring|\mathbb{Z}_p} \GL_{#1}}%
}
\newcommand{\ELgalois}{\mathscr{I}}
\newcommand{\weylorbit}[1]{\langle #1 \rangle}
\newcommand{\generalhodge}{\underline}
\newcommand{\deformrings}[1][{}]{\textbf{C}_{\witt #1}}
\newcommand{\univ}{\text{univ}}
\newcommand{\ord}{\text{ord}}
\newcommand{\can}{\text{can}}

\newcommand{\ELf}{\mathfrak{f}}

\newcommand{\charzero}{\mathscr}
\newcommand{\shimuradatum}{(\charzero{G}, \mathfrak{H})}
\newcommand{\reflexfield}{E}
\newcommand{\place}{v}
\newcommand{\genfield}{K}
\newcommand{\reflexring}{\integerring_{\reflexfield}}
\newcommand{\localring}{\integerring_{\reflexfield, \place}}
\newcommand{\levelatp}{{\charzero{K}_p}}
\newcommand{\levelawayp}{{\charzero{K}^p}}
\newcommand{\level}{{\charzero{K}}}

\newcommand{\reflexresiduefield}{\kappa(v)}
\newcommand{\integralmodel}{\mathscr{S}}
\newcommand{\univabelsch}{\mathscr{A}}
\newcommand{\pisog}{\mathscr{J}}
\newcommand{\pisoggenfiber}{J}
\newcommand{\pisogspecialfiber}{\pisog_0}
\newcommand{\frobcorrsection}{\phi}
\newcommand{\frobcorr}{\Phi}
\newcommand{\heckealg}[1]{\mathcal{H}(#1, \Q)}
\newcommand{\heckealgsupp}[1]{\mathcal{H}_0(#1, \Q)}
\newcommand{\twistedsatake}{\dot{\mathcal{S}}^G_\levi}
\newcommand{\tatemodule}[1]{T_p(#1)}
\newcommand{\rattatemodule}[1]{V_p(#1)}
\newcommand{\candeform}[1]{\generalhodge{\deform{#1}}^\can}

\makeatletter

\theoremstyle{theorem}
\newtheorem{theorem}[subsubsection]{Theorem}
\newtheorem{lemma}[subsubsection]{Lemma}
\newtheorem{prop}[subsubsection]{Proposition}
\newtheorem{cor}[subsubsection]{Corollary}

\newcounter{introtheorem}
\newtheorem{theorem*}[introtheorem]{Theorem}
\newtheorem{lemma*}[introtheorem]{Lemma}

\theoremstyle{definition}
\newtheorem{example}[subsubsection]{Example}

\newenvironment{remark}[1][Remark.]{\begin{trivlist}
\item[\hskip \labelsep {\bfseries #1}]}{\end{trivlist}}

\numberwithin{equation}{subsubsection}

\makeatother

\begin{document}

\tikzset{
    node style sp/.style={draw,circle,minimum size=\myunit},
    node style ge/.style={circle,minimum size=\myunit},
    arrow style mul/.style={draw,sloped,midway,fill=white},
    arrow style plus/.style={midway,sloped,fill=white},
}

\title{On the Hodge-Newton Filtration for $p$-divisible Groups of Hodge Type}
\author{Serin Hong}
\address{Department of Mathematics, California Institute of Technology}
\email{shong2@caltech.edu}

\maketitle

\rhead{}

\chead{}

\begin{abstract}
A $p$-divisible group, or more generally an $F$-crystal, is said to be \emph{Hodge-Newton reducible} if its Hodge polygon passes through a break point of its Newton polygon. Katz proved that Hodge-Newton reducible $F$-crystals admit a canonical filtration called the \emph{Hodge-Newton filtration}. The notion of Hodge-Newton reducibility plays an important role in the deformation theory of $p$-divisible groups; the key property is that the Hodge-Newton filtration of a $p$-divisible group over a field of characteristic $p$ can be uniquely lifted to a filtration of its deformation.

We generalize Katz's result to $F$-crystals that arise from an unramified local Shimura datum of Hodge type. As an application, we give a generalization of Serre-Tate deformation theory for local Shimura data of Hodge type. We also apply our deformation theory to study some congruence relations on Shimura varieties of Hodge type. 
\end{abstract}

\tableofcontents

\section{Introduction}

The motivation of this study is to generalize Serre-Tate deformation theory to $p$-divisible groups with additional structures that arise in Shimura varieties of Hodge type. The classical Serre-Tate deformation theory states that, if $\BT$ is an ordinary $p$-divisible group over a perfect field $k$ of characteristic $p>0$, its formal deformation space has a canonical structure of a formal torus over $\witt(k)$, the ring of Witt vectors over $k$. As a consequence, we get a canonical lifting $\BT^\can$ over $\witt(k)$ corresponding to the identity section of the formal torus. When $k$ is finite, $\BT^\can$ can be characterized as the unique deformation of $\BT$ to which all endomorphisms of $\BT$ lift. These results first appeared in the Woods Hole reports of Lubin, Serre and Tate \cite{Lubin-Serre-Tate64}.

The classical Serre-Tate deformation theory is based on the fact that an ordinary $p$-divisible group over $k$ admits a canonical filtration, called the slope filtration, which can be uniquely lifted to $\witt(k)$. For general $p$-divisible groups, this is no longer true; the slope filtration is given only up to an isogeny, and it does not necessarily lift to $\witt(k)$. Still, one can try to study their deformations by finding a canonical filtration which can be uniquely lifted to $\witt(k)$. For example, Messing in \cite{Messing72} proved that the multiplicative-bilocal-\'etale filtration of a $p$-divisible group over $k$ can be uniquely lifted to $\witt(k)$.

In \cite{Katz79}, Katz identified a large class of objects in the category of $F$-crystals which admit such a filtration. We say that an $F$-crystal $\dieudonnemodule$ over $k$ is \emph{Hodge-Newton reducible} if its Hodge polygon passes through a break point of its Newton polygon. A specified contact point divides the Newton polygon into two parts $\newtonmap_1$ and $\newtonmap_2$ where the slopes of $\newtonmap_1$ are less than the slopes of $\newtonmap_2$, and similarly the Hodge polygon into two parts $\mu_1$ and $\mu_2$. A \emph{Hodge-Newton decomposition} of $\dieudonnemodule$  is a decomposition of the form
\[ \dieudonnemodule = \dieudonnemodule_1 \oplus \dieudonnemodule_2\]
such that the Newton (resp. Hodge) polygon of $\dieudonnemodule_i$ is $\newtonmap_i$ (resp. $\mu_i)$ for $i=1, 2$. Such a decomposition induces a filtration
\[ 0 \subset \dieudonnemodule_1 \subset \dieudonnemodule\]
such that $\dieudonnemodule/\dieudonnemodule_1 = \dieudonnemodule_2$; this filtration is referred to as a \emph{Hodge-Newton filtration} of $\dieudonnemodule$. Katz proved that every Hodge-Newton reducible $F$-crystal over $k$ admits a Hodge-Newton decomposition. For $F$-crystals that arise from a $p$-divisible group, the Hodge-Newton filtration coincides with the multiplicative-bilocal-\'etale filtration.

In this paper we extend Katz's result to $p$-divisible groups and $F$-crystals that arise from an unramified local Shimura datum of Hodge type. In the special case of $\mu$-ordinary $p$-divisible groups which replace ordinary $p$-divisible groups in this setting, our result yields a unique lifting of the slope filtration and consequently leads to a generalization of Serre-Tate deformation theory. We also study certain congruence relations on Shimura varieties of Hodge type using our generalization of Serre-Tate deformation theory.

As another application of our result, the author proved Harris-Viehmann conjecture for $l$-adic cohomology of Rapoport-Zink spaces of Hodge type under the Hodge-Newton reducibility assumption in \cite{Hong16}.

We remark on previously known results for $p$-divisible groups and $F$-crystals of PEL type. For $\mu$-ordinary $p$-divisible groups of PEL type, Moonen in \cite{Moonen04} proved the unique lifting of the slope filtration, and used it to generalize Serre-Tate deformation theory to Shimura varieties of PEL type. Moonen also applied this deformation theory to study some congruence relations on Shimura varieties of PEL type. Existence of the Hodge-Newton decomposition for general PEL cases is due to Mantovan and Viehmann in \cite{Mantovan-Viehmann10}. They also proved the unique lifting of the Hodge-Newton filtration under some additional assumptions, which were later removed by Shen in \cite{Shen13}. Mantovan in \cite{Mantovan08} and Shen in \cite{Shen13} used these results to verify Harris-Viehmann conjecture in this context.

Let us now explain our results in more detail. Assume that $k$ is algebraically closed of characteristic $p$. Let $\witt$ be the ring of Witt vectors over $k$, and let $\quotientwitt$ be its quotient field. Let $\sigma$ denote the Frobenius automorphism over $k$, and also its lift to $\witt$ and $\quotientwitt$. We will consider an unramified local Shimura datum of Hodge type $(G, [b], \{\mu\})$, which consists of an unramified connected reductive group $G$ over $\Q_p$, a $\sigma$-conjugacy class $[b]$ of $G(\quotientwitt)$ and a conjugacy class of cocharacters $\{\mu\}$ of $G(\witt)$ satisfying certain conditions (see \ref{BT with tensors} for details). Since $G$ is unramified, we can choose its reductive model over $\Z_p$, which we will also by $G$. We also fix an embedding $G \hookrightarrow \GL(\module)$ for some finite free $\Z_p$-module $\module$. 
With a suitable choice of $b \in [b]$, our Shimura datum gives rise to an $F$-crystal $\dieudonnemodule$ over $k$ with additional structures determined by the choice of an embedding $G \hookrightarrow \GL(\module)$. When $\{\mu\}$ is minuscule, we also get a $p$-divisible group $\BT$ over $k$ which corresponds to $\dieudonnemodule$ via Dieudonn\'e theory. We will write $\generalhodge{\dieudonnemodule}$ (resp. $\generalhodge{\BT}$) for $\dieudonnemodule$ (resp. $\BT$) endowed with additional structures.

To the local Shimura datum $(G, [b], \{\mu\})$ (and also to $\generalhodge{\dieudonnemodule}$ and $\generalhodge{\BT}$), we associate two invariants, called the Newton point and the $\sigma$-invariant Hodge point, as defined by Kottwitz in \cite{Kottwitz85}. When $G= \GL_n$ or EL/PEL type, these invariants can be interpreted as convex polygons with rational slopes; for $G=\GL_n$, these polygons agree with the classical Newton polygon and Hodge polygon. For general group $G$, however, these invariants do not necessarily have an interpretation as polygons. Therefore, the notion of Hodge-Newton reducibility for general local Shimura data is defined in terms of group theoretic language, with respect to a specified parabolic subgroup $\parabolic \subsetneq G$ and its Levi factor $\levi$.

Our strategy is to study Hodge-Newton reducible local Shimura data using the previously studied cases $G= \GL_n$ or EL/PEL type. The main technical challenge is that the notion of Hodge-Newton reducibility is not functorial. For example, for a Hodge-Newton reducible unramified local Shimura datum of Hodge type $(G, [b], \{\mu\})$, the datum $(\GL(\module), [b], \{\mu\})$ obtained via the embedding $G \hookrightarrow \GL(\module)$ is not necessarily Hodge-Newton reducible if $G$ is not split. We overcome this obstacle by proving the following lemma:

\begin{lemma*} There exists a group $\EL{G}$ of EL type with the following properties:
\begin{enumerate}[label=(\roman*)]
\item the embedding $G \hookrightarrow \GL(\module)$ factors through $\EL{G}$,
\item if $(G, [b], \{\mu\})$ is Hodge-Newton reducible with respect to a parabolic subgroup $\parabolic \subsetneq G$ and its Levi factor $\levi$, then the datum $(\EL{G}, [b], \{\mu\})$ is Hodge-Newton reducible with respect to a parabolic subgroup $\EL{\parabolic} \subsetneq \EL{G}$ and its Levi factor $\EL{\levi}$ such that $\parabolic = \EL{\parabolic} \cap G$ and $\levi = \EL{\levi} \cap G$. 
\end{enumerate}
\end{lemma*}

For simplicity, we may assume that $\EL{G}= \ELgroupintegral$ where $\integerring$ is the ring of integer for some finite unramified extension $\unramext$ of $\Q_p$. When $(G, [b], \{\mu\})$ is Hodge-Newton reducible with respect to a parabolic subgroup $\parabolic \subsetneq G$ and its Levi factor $\levi$, we can choose an element $b \in [b] \cap \levi(\quotientwitt)$ and a representative $\mu \in \{\mu\}$ which factors through $\levi$. The above lemma yields a Levi subgroup $\EL{\levi} \subsetneq \EL{G}$, which is of the form
\[ \EL{\levi} = \ELgroupintegral[n_1] \times \cdots \times \ELgroupintegral[n_r].\]
For $j=1, 2, \cdots, r$, we denote by $\EL{\levi}_j$ the $j$-th factor in the above decomposition, and by $\levi_j$ the image of $\levi = \EL{\levi} \cap G$ under the projection $\EL{\levi} \twoheadrightarrow \EL{\levi}_j$.  Then the datum $(G, [b], \{\mu\})$ induces local Shimura data $(\levi_j, [b_j], \{\mu_j\})$ via the projections $\levi \twoheadrightarrow \levi_j$.

Our first main result is existence of the Hodge-Newton decomposition in this setting. For $p$-divisible groups with additional structures, the theorem can be stated as follows:

\begin{theorem*}\label{intro theorem hn decomp} Assume that $(G, [b], \{\mu\})$ is Hodge-Newton reducible with respect to a parabolic subgroup $\parabolic \subsetneq G$ and its Levi factor $\levi$. Let $\generalhodge{\BT}$ be a $p$-divisible group over $k$ with additional structures corresponding to the choice $b \in [b] \cap \levi(\quotientwitt)$. Consider the local Shimura data $(\levi_j, [b_j], \{\mu_j\})$ as explained above. Then $\generalhodge{\BT}$ admits a decomposition 
\[ \generalhodge{\BT} = \generalhodge{\BT}_1 \times \cdots \times \generalhodge{\BT}_r\]
where $\generalhodge{\BT}_j$ is a $p$-divisible group over $k$ with additional structures that arises from the datum $(\levi_j, [b_j], \{\mu_j\})$.
\end{theorem*}
We emphasize that this result also applies to $F$-crystals with additional structures, as our argument does not require $\{\mu\}$ to be minuscule.

The Hodge-Newton decomposition of $\generalhodge{\BT}$ in Theorem \ref{intro theorem hn decomp} induces the Hodge-Newton filtration of $\generalhodge{\BT}$
\[ 0 \subset {\BT}^{(r)} \subset {\BT}^{(r-1)} \subset \cdots \subset {\BT}^{(1)} = {\BT}\]
where ${\BT}^{(j)} = {\BT}_j \times \cdots \times {\BT}_r$ for $j=1, 2, \cdots, r$. Note that each quotient $\BT^{(j+1)}/\BT^{(j)} \simeq \BT_j$ admits additional structures that arise from the datum $(\levi_j, [b_j], \{\mu_j\})$. Our second result is the unique lifting of the Hodge-Newton filtration to deformation rings. 

\begin{theorem*} \label{intro theorem hn filtration} Retain the notations in Theorem \ref{intro theorem hn decomp}. In addition, we assume that $p>2$. Let $\generalhodge{\deform{\BT}}$ be a deformation of $\generalhodge{\BT}$ over a formally smooth $\witt$-algebra $\genring$ of the form $\genring = \witt[[u_1, \cdots, u_N]]$ or $\genring = \witt[[u_1, \cdots, u_N]]/(p^m)$. Then ${\deform{\BT}}$ admits a unique filtration
\[ 0 \subset \deform{{\BT}}^{(r)} \subset \deform{{\BT}}^{(r-1)} \subset \cdots \subset \deform{{\BT}}^{(1)} = \deform{{\BT}} \]
with the following properties:
\begin{enumerate}[label=(\roman*)]
\item each $\deform{\BT}^{(j)}$ is a deformation of $\BT^{(j)}$ over $\genring$ (without additional structures),  
\item each $\generalhodge{\deform{\BT}^{(j)} / \deform{\BT}^{(j+1)}}$ is a deformation of $\generalhodge{\BT}_j$ over $\genring$ (with additional structures). 
\end{enumerate}
\end{theorem*}

An important case is when $\generalhodge{\BT}$ is $\mu$-ordinary, i.e., the Newton point and the $\sigma$-invariant Hodge point of $\generalhodge{\BT}$ coincide. In this case, Theorem \ref{intro theorem hn decomp} gives us a ``slope decomposition"
\[ \generalhodge{\BT} = \generalhodge{\BT}_1 \times \generalhodge{\BT}_2 \times \cdots \times \generalhodge{\BT}_r.\]
Then Theorem \ref{intro theorem hn filtration} implies that the induced ``slope filtration" can be uniquely lifted to a filtration of a deformation of $\generalhodge{\BT}$. As a result, we find a generalization of Serre-Tate deformation theory. When $r=2$, the theorem can be stated as follows: 

\begin{theorem*}\label{intro serre-tate theory for two slopes}
Assume that $\generalhodge{\BT}$ is $\mu$-ordinary with two factors in its slope decomposition. If $p>2$, the formal deformation space $\Def_{\BT, G}$ of $\generalhodge{\BT}$ has a natural structure of a $p$-divisible group over $\witt$. More precisely, there exist two positive integers $h$ and $d$ (which can be explicitly computed) such that
\[ \Def_{\BT, G} \cong \EL{\deform{Y}}_h^{d}\]
as $p$-divisible groups over $\witt$, where $\EL{\deform{Y}}_h$ is the Lubin-Tate formal group of height $h$. 
\end{theorem*}

As an application of our deformation theory, we prove the following congruence relation on Shimura varieties of Hodge type:

\begin{theorem*} \label{intro congruence relation on shimura variety of hodge type}
Let $\shimuradatum$ be a Shimura datum of Hodge type. Let $\Phi$ denote the Frobenius correspondence on the associated Shimura variety in characteristic $p$, i.e., the special fiber of the associated integral model. Then we have a congruence relation $H_{\shimuradatum}(\Phi) = 0$ over the $\mu$-ordinary locus, where $H_{\shimuradatum}$ is the Hecke polynomial associated to the datum $\shimuradatum$. 
\end{theorem*}

We remark that, after this paper was submitted, Shankar and Zhou in \cite{Shankar-Zhou16} independently obtained a similar generalization of Serre-Tate deformation theory using a different method.

We now give a brief description of the structure of this paper. In section 2, we recall some basic definitions, such as $F$-isocrystals with $G$-structure and unramified local Shimura data of Hodge type, and review Faltings's explicit construction of the ``universal deformation" of $p$-divisible groups with additional structures. In section 3, we define and  study the notion of Hodge-Newton reducibility for unramified local Shimura data of Hodge type (Theorem \ref{intro theorem hn decomp} and Theorem \ref{intro theorem hn filtration}). In section 4, we establish a generalization of Serre-Tate deformation theory for local Shimura data of Hodge type (e.g. Theorem \ref {intro serre-tate theory for two slopes}). In section 5, we briefly review the Newton stratification on the Shimura varieties of Hodge type and study some congruence relations on the $\mu$-ordinary locus (Theorem \ref{intro congruence relation on shimura variety of hodge type}).

\subsection*{Acknowledgments} I would like to sincerely thank my advisor E. Mantovan for her continuous encouragement and advice. I also thank T. Wedhorn for his helpful comments on a preliminary version of this paper.

\section{Preliminaries}

\subsection{Group theoretic notations} $ $

\subsubsection{}\label{notations on Weil restriction and Frobenius actions}
Throughout this paper, $k$ is a perfect field of positive characteristic $p$. We write $\witt(k)$ for the ring of Witt vectors over $k$, and $\quotientwitt(k)$ for its quotient field. We will often write $\witt = \witt(k)$ and $\quotientwitt = \quotientwitt(k)$. We generally denote by $\sigma$ the Frobenius automorphism over $k$, and also its lift to $\witt(k)$ and $\quotientwitt(k)$. 

Let $\module$ be a finitely generated free module over $\Z_p$. Then $\sigma$ acts on $\module_\witt = \module \otimes_{\Z_p} \witt$ and on $\GL(\module_\witt) = \GL(\module) \otimes_{\Z_p} W$ via $1 \otimes \sigma$. Alternatively, we may write this action as $\sigma(g) = (1 \otimes \sigma) \circ g \circ (1 \otimes \sigma\inv)$ for $g \in \GL(\module_\witt)$. We also have an induced action of $\sigma$ on the group of cocharacters $\Hom_\witt(\Gm, \GL(\module_\witt))$ defined by $\sigma(\mu)(a) = \sigma(\mu(a))$. 

For two $\Z_p$-algebras $\genring \subseteq \genring'$, we will denote by $\Res_{\genring'|\genring} \GL_n$ the Weil restriction of $\GL_n \otimes_{\genring} \genring'$. If $\integerring$ is a finite unramified extension of $\Z_p$, a choice of $\sigma$-invariant basis of $\integerring$ over $\Z_p$ determines an embedding of affine $\Z_p$-groups
\[ \Res_{\integerring|\Z_p}\GL_n \hookrightarrow \GL_{mn},\]
where $m = |\integerring : \Z_p|$. If $\module$ is a free module over $\integerring$ of rank $n$, then there is a natural identification $\Res_{\integerring|\Z_p} \GL(\module) \otimes_{\Z_p} \witt \cong \GL_{\integerring \otimes_{\Z_p} \witt}(\module_\witt)$ where the latter is identified with a product of $m$ copies of $\GL_n \otimes_{\Z_p} \witt$ after choosing a $\sigma$-invariant basis of $\integerring$ over $\Z_p$.

\subsubsection{}\label{def of unramified group}
Let $G$ be a connected reductive group over $\Q_p$ with a Borel subgroup $\borel \subseteq G$ and a maximal torus $T \subseteq G$. We will write $(\characters(T), \roots, \cocharacters(T), \coroots)$ for the associated root datum, and $\weylgroup$ for the associated Weyl group. The choice of $\borel$ determines a set of positive roots $\roots^+ \subseteq \roots$ and a set of positive coroots $\coroots{}^+ \subseteq \coroots$. The group $\weylgroup$ naturally acts on $\cocharacters(T)$ (resp. $\characters(T)$), and the dominant cocharacters (resp. dominant characters) form a full set of representatives for the orbits in $\cocharacters(T) / \weylgroup$ (resp. $\characters(T)/\weylgroup$).

Except for \ref{F-isocrystals with G-structure}, we will always assume that $G$ is \emph{unramified}. This means that $G$ satisfies the following equivalent conditions:
\begin{enumerate}[label=(\roman*)]
\item $G$ is quasi-split and split over a finite unramified extension of $\Q_p$. 
\item $G$ admits a reductive model over $\Z_p$. 
\end{enumerate}
When $G$ is unramified, we fix a reductive model $G_{\Z_p}$ over $\Z_p$, and will often write $G = G_{\Z_p}$ if there is no risk of confusion. We also fix a Borel subgroup $B \subseteq G$ and a maximal torus $T \subseteq G$ which are both defined over $\Z_p$.

For any local, strictly Henselian $\Z_p$ algebra $R$ and a cocharacter $\mu: \Gm \to G_R$, we denote by $\{\mu\}$ the $G(R)$-conjugacy class of $\mu$. We have identifications $\weylgroup \cong N_G(T)(R)/T(R)$ and $\cocharacters(T) \cong \Hom_R(\Gm, T_R)$, which induce a bijection between $\cocharacters(T)/\weylgroup$ and the set of $G(R)$-conjugacy classes of cocharacters for $G_R$.  We will be mostly interested in the case $R=\witt(k)$ for some algebraically closed $k$, where we also have a bijection
\[ \Hom_\witt(\Gm, G_\witt)/G(\witt) \cong \Hom_{\quotientwitt}(\Gm, G_{\quotientwitt})/G(\quotientwitt) \stackrel{\sim}{\rightarrow} G(\witt) \backslash G(\quotientwitt)/ G(\witt)\]
induced by $\{\mu\} \mapsto G(\witt) \mu(p) G(\witt)$; indeed, the first bijection follows from the fact that $G$ is split over $W$, while the second bijection is the Cartan decomposition.

\subsection{$F$-isocrystals with $G$-structure}\label{F-isocrystals with G-structure}$ $

We review the theory of $F$-isocrystals with $G$-structure due to R. Kottwitz in \cite{Kottwitz85} and  \cite{Kottwitz97}. We do not assume that $G$ is unramified for this subsection.

\subsubsection{}
Let $k$ be a perfect field of positive characteristic $p$. An $F$-isocrystal over $k$ is a vector space $V$ over $\quotientwitt(k)$ with an isomorphism $F: \sigma^* V \stackrel{\sim}{\rightarrow} V$. The dimension of $V$ is called the \emph{height} of the isocrystal. Let $F\text{-Isoc}(k)$ denote the category of $F$-isocrystals over $k$. For a connected reductive group $G$ over $\Q_p$, we define an \emph{$F$-isocrystal over $k$ with $G$-structure} as an exact faithful tensor functor
\[\text{Rep}_{\Q_p}(G) \rightarrow F\text{-Isoc}(k).\]

\begin{example}\label{F-isocrystal examples} (i) An $F$-isocrystal with $\GL_n$-structure is an $F$-isocrystal of height $n$. 

(ii) If $G = \Res_{\unramext|\Q_p} \GL_n$ where $\unramext|\Q_p$ is a finite extension of degree $m$, an $F$-isocrystal with $G$-structure is an $F$-isocrystal $V$ of height $mn$ together with a $\Q_p$-homomorphism $\iota: \unramext \to \End_k(V)$. 

(iii) If $G = \text{GSp}_{2n}$, an $F$-isocrystal with $G$-structure is an $F$-isocrystal $V$ of height $2n$ together with a non-degenerate alternating pairing $V \otimes V \to \mathbf{1}$, where $\mathbf{1}$ is the unit object of the tensor category $F\text{-Isoc}(k)$. 

\end{example}

\subsubsection{} \label{classification of F-isocrystals with G-structure by B(G)}
Let us now assume that $k$ is algebraically closed. We say that $b, b' \in G(\quotientwitt)$ are $\sigma$-conjugate if there exists $g \in G(\quotientwitt)$ such that $b' = g b \sigma(g)\inv$. We denote by $B(G)$ the set of all $\sigma$-conjugacy classes in $G(\quotientwitt)$. The definition of $B(G)$ is independent of $k$ in the sense that any inclusion $k \hookrightarrow k'$ into another algebraically closed field of characteristic $p$ induces a bijection between the $\sigma$-conjugacy classes of $G(\quotientwitt(k))$ and those of $G(\quotientwitt(k'))$. We will write $[b]_G$, or simply $[b]$ when there is no risk of confusion, for the $\sigma$-conjugacy class of $b \in G(\quotientwitt)$.

The set $B(G)$ classifies the $F$-isocrystals over $k$ with $G$-structure up to isomorphism. We describe this classification as explained in \cite{Rapoport-Richartz96}, 3.4. Given $b \in G(\quotientwitt)$ and a $G$-representation $(V, \rho)$ over $\Q_p$, set $N_b(\rho)$ to be $V \otimes_{\Q_p} \quotientwitt$ with a $\sigma$-linear automorphism $F=\rho(b) \circ (1 \otimes \sigma)$. Then $N_b: \text{Rep}_{\Q_p}(G) \rightarrow F\text{-Isoc}(k)$ is an exact faithful tensor functor. It is evident that two elements $b_1, b_2 \in G(\quotientwitt)$ give an isomorphic functor if and only if they are $\sigma$-conjugate. One can also prove that any $F$-isocrystal on $k$ with $G$-structure is isomorphic to a functor $N_b$ for some $b \in G(\quotientwitt)$. Hence the association $b \mapsto N_b$ induces the desired classification.

\subsubsection{} \label{def of newton set}
Let $\mathbb{D}$ be the pro-algebraic torus with character group $\Q$. We introduce the set
\[ \newtonset(G) := (\text{Int } G(\quotientwitt) \backslash \Hom_{\quotientwitt}(\mathbb{D}, G))^{\langle \sigma \rangle}.\]
If we fix a Borel subgroup $B \subseteq G$ and a maximal torus $T \subseteq G$, we can also write
\[ \newtonset(G) = (\cocharacters(T)_\Q / \weylgroup)^{\langle \sigma \rangle}.\]

We can define a partial order $\preceq$ on $\newtonset(G)$ as follows. Let $\bar{C}$ be the closed Weyl chamber. First we define a partial order $\preceq_1$ on $\cocharacters(T)_\R$ by declaring that $\alpha \preceq_1 \alpha'$ if and only if $\alpha' - \alpha$ is a nonnegative linear combination of positive coroots. Each orbit in $\cocharacters(T)_\R/\weylgroup$ is represented by a unique element in $\bar{C}$, so the restriction of $\preceq_1$ to $\bar{C}$ induces a partial order $\preceq_2$ on $\cocharacters(T)_\R/\weylgroup$. Then we take $\preceq$ to be the restriction of $\preceq_2$ to $(\cocharacters(T)_\Q/\weylgroup)^{\langle \sigma \rangle}$. 

\begin{remark} A closed embedding $G_1 \hookrightarrow G_2$ of connected reductive algebraic groups over $\Q_p$ induces an order-preserving map $\newtonset(G_1) \to \newtonset(G_2)$, which is not necessarily injective. \end{remark}

\subsubsection{} \label{def of newton map and kottwitz map} Kottwitz studied the set $B(G)$ by introducing two maps
\[ \newtonmap_G: B(G) \rightarrow \newtonset(G), \quad \kottwitzmap_G  : B(G) \rightarrow \pi_1(G)_{\langle \sigma \rangle} \]
called the Newton map and the Kottwitz map of $G$. We refer the readers to \cite{Kottwitz85}, \S 4 or \cite{Rapoport-Richartz96}, \S1 for definition of the Newton map, and \cite{Kottwitz97}, \S4 and \S7 for definition of the Kottwitz map. Both maps are functorial in $G$; more precisely, they induce natural transformations of set-valued functors on the category of connected reductive groups 
\[ \newtonmap: B(\cdot) \rightarrow \newtonset(\cdot), \quad \kottwitzmap: B(\cdot) \rightarrow \pi_1(\cdot)_{\langle \sigma \rangle}. \]

Given $[b] \in B(G)$ (and its corresponding $F$-isocrystal with $G$-structure), we will often refer to two invariants $\newtonmap_G([b])$ and $\kottwitzmap_G([b])$  respectively as the \emph{Newton point} and the \emph{Kottwitz point} of $[b]$. 
Kottwitz proved that a $\sigma$-conjugacy class is determined by its Newton point and Kottwitz point; in other words, the map
\[ \newtonmap_G \times \kottwitzmap_G : B(G) \rightarrow \newtonset(G) \times \pi_1(G)_{\langle \sigma \rangle}\]
is injective (\cite{Kottwitz97}, 4.13).

\begin{example}\label{newton map for GLn} We describe the Newton map for $G= GL_n$. Let $T$ be the diagonal torus contained in the Borel subgroup of lower triangular matrices. Then using the identification $\cocharacters(T) \cong \Z^n$ we can write 
\[\newtonset(\GL_n) = \{ (r_1, r_2, \cdots, r_n) \in \Q^n:  r_1 \leq r_2 \leq \cdots \leq r_n\},\]
which can be identified with the set of convex polygons with rational slopes. We have $(r_i) \preceq (s_i)$ if and only if $\displaystyle \sum_{i=1}^l (r_i - s_i) \geq 0$ for all $l \in \{1, 2, \cdots, n\}$, so the ordering $\preceq$ coincides with the usual ``lying above" order for convex polygons.

If $V$ is an $F$-isocrystal $V$ of height $n$ associated to $[b] \in B(\GL_n)$, its Newton point $\newtonmap_{\GL_n}([b])$ is the same as its classical Newton polygon. In this case, the Kottwitz point $\kottwitzmap_{\GL_n}([b])$ is determined by the Newton point $\newtonmap_{\GL_n}([b])$. Hence $V$ and $[b]$ are determined by the Newton point $\newtonmap_{\GL_n}([b])$, and we recover Manin's classification of $F$-isocrystals by their Newton polygons in \cite{Manin63}. 
\end{example}

\subsubsection{}\label{mu-ordinariness for isocrystals}
Let $\mu \in \cocharacters(T)$ be a dominant cocharacter. Then $\mu$ represents a unique conjugacy class of cocharacters of $G(\quotientwitt)$ which we denote by $\{\mu\}$. We identify $\mu$ with its image in $\cocharacters(T)/\weylgroup$, and define
\[ \bar{\mu} = \dfrac{1}{m} \sum_{i=0}^{m-1} \sigma^i(\mu) \in \newtonset(G)\]
where $m$ is some integer such that $\sigma^m(\mu) = \mu$. We also let $\mu^\natural \in \pi_1(G)_{\langle \sigma \rangle}$ be the image of $\mu$ under the natural projection $\cocharacters(T) \rightarrow \pi_1 (G)_{\langle \sigma \rangle} = (\cocharacters(T) / \langle \alpha^\vee : \alpha^\vee \in \Phi^\vee \rangle)_{\langle \sigma \rangle}$. The characterization of the Newton map in \cite{Kottwitz85}, 4.3 shows that $\bar{\mu}$ is the image of $[\mu(p)]$ under $\newtonmap_G$.  It also follows directly from the definition of $\kottwitzmap_G$ that $\mu^\natural$ is the image of $[\mu(p)]$ under $\kottwitzmap_G$.

Let us now define the set
\[ B(G, \{\mu\}) := \{ [b] \in  B(G) : \kottwitzmap_G([b]) = \mu^\natural, \newtonmap_G([b]) \preceq \bar{\mu} \}. \]
This set is known to be finite (see \cite{Rapoport-Richartz96}, 2.4.). It is also non-empty since we have $[\mu(p)] \in B(G, \{\mu\})$ by the discussion in the previous paragraph. 

Since the Newton map is injective on $B(G, \{\mu\})$ (see \ref{def of newton map and kottwitz map}), the partial order $\preceq$ on $\newtonset(G)$ induces a partial order on $B(G, \{\mu\})$. We will also use the symbol $\preceq$ to denote this induced partial order. Note that $[\mu(p)]$ is a unique maximal element in $B(G, \{\mu\})$ as the inequality $[b] \preceq [\mu(p)]$ clearly holds for all $[b] \in B(G, \{\mu\})$.

We refer to the $\sigma$-conjugacy class $[\mu(p)]$ as the \emph{$\mu$-ordinary} element of $B(G, \{\mu\})$. We say that an $F$-isocrystal over $k$ with $G$-structure is $\mu$-ordinary if it corresponds to $[\mu(p)]$ in the sense of \ref{classification of F-isocrystals with G-structure by B(G)}. Note that a $\sigma$-conjugacy class $[b] \in B(G, \{\mu\})$ is $\mu$-ordinary if and only if $\newtonmap_G([b]) = \bar{\mu}$.

\subsection{Unramified local Shimura data of Hodge type}\label{BT with tensors}$ $

In this subsection, we review the notion of unramified local Shimura data of Hodge type and describe $F$-crystals with additional structures that arise from such data.

\subsubsection{} \label{definition of local shimura datum of hodge type}

Assume that $k$ is algebraically closed. By an \emph{unramified (integral) local Shimura datum of Hodge type}, we mean a tuple $(G, [b], \{\mu\})$ where
\begin{itemize}
\item $G$ is an unramified connected reductive group over $\Q_p$;
\item $[b]$ is a $\sigma$-conjugacy class of $G(\quotientwitt)$;
\item $\{\mu\}$ is a $G(\witt)$-conjugacy class of cocharacters of $G$,
\end{itemize}
which satisfy the following two conditions:
\begin{enumerate}[label=(\roman*)]
\item\label{unram loc shimura datum condition 1} $[b] \in B(G, \{\mu\})$,
\item\label{unram loc shimura datum condition 2} there exists a faithful $G$-representation $\module \in \text{Rep}_{\Z_p}(G)$ (with its dual $\module^*$) such that, for all $b \in [b]$ and $\mu \in \{\mu\}$ satisfying $b \in G(\witt) \mu(p) G(\witt)$, we have a $\witt$-lattice 
\[\dieudonnemodule \simeq \module^* \otimes_{\Z_p} \witt \subset  N_b(\module^* \otimes_{\Z_p} \Q_p)\] 
with the property $p \dieudonnemodule \subset F \dieudonnemodule \subset \dieudonnemodule$. 
\end{enumerate}
Here $N_b: \text{Rep}_{\Q_p}(G) \rightarrow F\text{-Isoc}(k)$ is the functor defined in \ref{classification of F-isocrystals with G-structure by B(G)} which is uniquely determined by $[b]$. The set $G(\witt) \mu(p) G(\witt)$ is independent of the choice $\mu \in \{\mu\}$ as explained in \ref{def of unramified group}. The property $p \dieudonnemodule \subset F \dieudonnemodule \subset \dieudonnemodule$ means that $\dieudonnemodule$ is an $F$-crystal over $k$ (with a $\sigma$-linear endomorphism $F$). 
The requirement $b \in G(\witt) \mu(p) G(\witt)$ ensures that the Hodge filtration of $\dieudonnemodule$ is induced by $\sigma\inv(\mu)$. 

In practice when one tries to check that a given tuple $(G, [b], \{\mu\})$ is an unramified local Shimura datum, it is often more convenient to work with the following equivalent conditions of \ref{unram loc shimura datum condition 1} and \ref{unram loc shimura datum condition 2}:
\begin{enumerate}[label=(\roman*')]
\item\label{unram loc shimura datum condition 1'} $[b] \cap G(\witt) \mu(p) G(\witt)$ is not empty for some (and hence for all) $\mu \in \{\mu\}$, 
\item\label{unram loc shimura datum condition 2'} there exists a faithful $G$-representation $\module \in \text{Rep}_{\Z_p}(G)$ (with its dual $\module^*$) such that, for \emph{some} $b \in [b]$ and $\mu \in \{\mu\}$ satisfying $b \in G(\witt) \mu(p) G(\witt)$, we have a $\witt$-lattice 
\[\dieudonnemodule \simeq \module^* \otimes_{\Z_p} \witt \subset  N_b(\module^* \otimes_{\Z_p} \Q_p)\] 
with the property $p \dieudonnemodule \subset F \dieudonnemodule \subset \dieudonnemodule$. 
\end{enumerate}
The equivalence of \ref{unram loc shimura datum condition 1} and \ref{unram loc shimura datum condition 1'} is due to work of several authors, including Kottwitz-Rapoport \cite{Kottwitz-Rapoport03}, Lucarelli \cite{Lucarelli04} and Gashi \cite{Gashi10}. Note that \ref{unram loc shimura datum condition 1'} ensures that the condition \ref{unram loc shimura datum condition 2} is never vacuously satisfied. To see the equivalence of \ref{unram loc shimura datum condition 2} and \ref{unram loc shimura datum condition 2'}, one observes that existence of $\dieudonnemodule$ is equivalent to the condition that the linearization of $F$ has an integer matrix representation after taking some $\sigma$-conjugate, which depends only on $[b]$.

\begin{remark}
When $\{\mu\}$ is minuscule, an unramified local Shimura datum of Hodge type as defined above is a local Shimura datum as defined by Rapoport and Viehmann in \cite{Rapoport-Viehmann14}, Definition 5.1. In fact, since $G$ is split over $\witt$, we may view geometric conjugacy classes of cocharacters as $G(\witt)$-conjugacy classes of cocharacters. 
\end{remark}

Using the conditions  \ref{unram loc shimura datum condition 1'} and \ref{unram loc shimura datum condition 2'} one easily verifies the following functorial properties of unramified local Shimura data of Hodge type:

\begin{lemma} \label{functoriality of unram local shimura data of hodge type} 
Let $(G, [b], \{\mu\})$ be an unramified local Shimura datum of Hodge type. 
\begin{enumerate} 
\item\label{product functoriality of unram local shimura data of hodge type} If $(G', [b'], \{\mu'\})$ is another unramified local Shimura datum of Hodge type, the tuple $(G \times G', [b, b'], \{\mu, \mu'\})$ is also an unramified local Shimura datum of Hodge type.
\item\label{hom functoriality of unram local shimura data of hodge type} For any homomorphism $f: G \longrightarrow G'$ of unramified connected reductive group defined over $\Z_p$, the tuple $(G', [f(b)], \{f \circ \mu\})$ is an unramified local Shimura datum of Hodge type. 
\end{enumerate}
\end{lemma}

\subsubsection{} \label{description of F-crystals with tensors}

For the rest of this section, we fix our unramified local Shimura datum of Hodge type $(G, [b], \{\mu\})$ and also a faithful $G$-representation $\module \in \text{Rep}_{\Z_p}(G)$ in the condition \ref{unram loc shimura datum condition 2} of \ref{definition of local shimura datum of hodge type}. By Lemma \ref{functoriality of unram local shimura data of hodge type}, we obtain a morphism of unramified local Shimura data of Hodge type
\[(G, [b], \{\mu\}) \longrightarrow (\GL(\module), [b]_{\GL(\module)}, \{\mu\}_{\GL(\module)}).\]

For a $\Z_p$-algebra $\genring$, we let $\module_\genring^\otimes$ denote the direct sum of all the $\genring$-modules which can be formed from $\module_\genring := \module \otimes_{\Z_p} \genring$ using the operations of taking duals, tensor products, symmetric powers and exterior powers. An element of $\module_\genring^\otimes$ is called a \emph{tensor} on $\module_\genring$. For the dual $\genring$-module $\module^*_\genring$ of $\module_\genring$, we can similarly define $(\module_\genring^*)^\otimes$ which has a natural identification $(\module_\genring^*)^\otimes = \module_\genring^\otimes$. An automorphism $f$ of $\module_\genring$ induces an automorphism $(f\inv)^*$ of $\module^*_\genring$ and thus an automorphism $f^\otimes$ of $\module_\genring^\otimes$. 

Let us now choose an element $b \in [b] \cap G(\witt) \mu(p) G(\witt)$ and take $\dieudonnemodule \simeq \module^* \otimes_{\Z_p} \witt$ as in the condition \ref{unram loc shimura datum condition 2} of \ref{definition of local shimura datum of hodge type}. A standard result by Kisin in \cite{Kisin10}, Proposition 1.3.2 gives a finite family of tensors $(s_i)_{i \in I}$ on $\module$ such that $G$ is the pointwise stabilizer of the $s_i$; i.e., for any $\Z_p$-algebra $\genring$ we have
\[ G(\genring) = \{ g \in \GL(\module_{\genring}) : g^\otimes((s_i)_{\genring}) = (s_i)_{\genring} \text{ for all } i \in I\}.\]
Hence $\dieudonnemodule \simeq \module^* \otimes_{\Z_p} \witt$ is equipped with tensors $(t_i) := (s_i \otimes 1)$, which are $F$-invariant since the linearization of $F$ on $\dieudonnemodule[1/p] = N_b(\module^* \otimes_{\Z_p} \Q_p)$ is given by an element $b \in G(\quotientwitt)$ in the conjugacy class $[b]$. We may regard the tensors $(t_i)$ as additional structures on $\dieudonnemodule$ induced by the group $G$. Following the terminology of \ref{F-isocrystals with G-structure}, we will often refer to these additional structures as \emph{$G$-structure}. We will write $\generalhodge{\dieudonnemodule}:=(\dieudonnemodule, (t_i))$ to indicate the $F$-crystal $\dieudonnemodule$ with $G$-structure.

For a $p$-divisible group $\BT'$ over a $\Z_p$-scheme $\genscheme$, we will write $\D(\BT')$ for its (contravariant) Dieudonn\'e module. When $\{\mu\}$ is minuscule, we have a unique $p$-divisible group $\BT$ over $k$ with $\D(\BT) = \dieudonnemodule$. In this case, we write $\generalhodge{\BT} :=(\BT, (t_i))$ to indicate the $p$-divisible group $\BT$ with $G$-structure.

\subsubsection{}\label{invariants of local shimura data}
For the datum $(G, [b], \{\mu\})$, we can define its Newton point and Kottwitz point by $\newtonmap_G([b])$ and $\kottwitzmap_G([b])$. Taking a unique dominant representative $\mu$ of $\{\mu\}$, we can also define $\bar{\mu}$ as in \ref{mu-ordinariness for isocrystals}, which we call the \emph{$\sigma$-invariant Hodge point} of $(G, [b], \{\mu\})$. We say that $(G, [b], \{\mu\})$ is $\mu$-ordinary if $[b]$ is $\mu$-ordinary.

For the $F$-crystal with $G$-structure $\generalhodge{\dieudonnemodule}$, we define its Newton point, Kottwitz point and $\sigma$-invariant Hodge point to be the corresponding invariants for $(G, [b], \{\mu\})$. We say that $\generalhodge{\dieudonnemodule}$ is ordinary if $(G, [b], \{\mu\})$ is ordinary. When $\{\mu\}$ is minuscule, these definitions obviously extend to the corresponding $p$-divisible group with $G$-structure $\generalhodge{\BT}$.

\begin{remark} We can further extend most of the notions defined in this section to the case when $k$ is not algebraically closed. For example, we may define an $F$-crystal over $k$ with $G$-structure as an $F$-crystal $\dieudonnemodule$ over $k$ equipped with tensors $(t_i)$ such that the pair $(\dieudonnemodule \otimes_{\witt(k)} \witt(\bar{k}), (t_i \otimes 1))$ is an $F$-crystal over $\bar{k}$ with $G$-structure as defined in \ref{description of F-crystals with tensors}. Then we have natural notions of the Newton point, Kottwitz point, $\sigma$-invariant Hodge point and $\mu$-ordinariness induced by the corresponding notions for $(\dieudonnemodule \otimes_{\witt(k)} \witt(\bar{k}), (t_i \otimes 1))$. This explains why we may safely focus our study on the case when $k$ is algebraically closed. \end{remark}

\begin{example} \label{BT with O-module structure} As a concrete example, let us consider the case $G = \ELgroupintegral$ where $\integerring$ is the ring of integers of some finite unramified extension of $\Q_p$.

Choosing a family of tensors $(s_i)$ on $\module$ whose pointwise stabilizer is $G$ amounts to choosing a $\Z_p$-basis of $\integerring$. Hence $\generalhodge{\dieudonnemodule}= (\dieudonnemodule, (t_i))$ can be identified with an $F$-crystal $\dieudonnemodule$ with an action of $\integerring$ (cf.  Example \ref{F-isocrystal examples}.(ii)). Following Moonen in \cite{Moonen04}, we will often say \emph{$\integerring$-module structure} in lieu of $G$-structure. 

We now take $\ELgalois := \Hom(\integerring, \witt(k))$ which is a cyclic group of order $m:= |E: \Q_p|$. For convenience, we will write $i + s := \sigma^s \circ i$ for any $i \in \ELgalois$ and $s \in \Z$. Then $\dieudonnemodule$, being a module over $\integerring \otimes_{\Z_p} \witt(k) = \prod_{i \in \ELgalois} \witt(k)$, decomposes into character spaces
\begin{equation}\label{char space decomp EL type} \dieudonnemodule = \bigoplus_{i \in \ELgalois} M_i \quad\quad \text{ where }  M_i = \{ x \in M : a \cdot x = i(a) x\}.\end{equation}
For each $i \in \ELgalois$, the Frobenius map $F$ restricts to a $\sigma$-linear map $F_i : \dieudonnemodule_i \to \dieudonnemodule_{i+1}$. Then the map $F^m$ restricts to a $\sigma^m$-linear endomorphism $\phi_i$ of $M_i$, thereby yielding a $\sigma^m$-$F$-crystal $(M_i, \phi_i)$ over $k$. By construction, $F_i$ induces an isogeny from $\sigma^*(M_i, \phi_i)$ to $(M_{i+1}, \phi_{i+1})$. This implies that the rank and the Newton polygon of $(M_i, \phi_i)$ is independent of $i \in \ELgalois$. We will write $d$ for the rank of $(M_i, \phi_i)$. 

The decomposition \eqref{char space decomp EL type} yields a decomposition
\[ M / FM = \bigoplus_{i \in \ELgalois} M_i/F_{i-1} M_i.\]
Define a function $\ELf : \ELgalois \to \Z$ by setting $\ELf(i)$ to be the rank of $M_i/F_{i-1} M_i$. We refer to the datum $(d, \ELf)$ as the \emph{type} of $\generalhodge{\dieudonnemodule}$.

Let us describe the Newton point in this setting. Using the identifications $G_\witt \cong \prod_{i \in \ELgalois} \GL(M_i)$ and $\cocharacters(T) \cong \Z^{md}$ we can write
\[ \cocharacters(T)_\Q / \weylgroup = \{ (x_1, \cdots, x_{md}) \in \Q^{md}:  x_{ds+1}  \leq \cdots \leq x_{d(s+1)} \text{ for } s=0, 1, \cdots, m-1 \}.\]
For $\mu = (x_1, \cdots, x_{md}) \in \cocharacters(T)_\Q / \weylgroup$ the action of $\sigma$ is given by $\sigma(\mu) = (y_1, \cdots, y_{md})$ where $y_t = x_{t+d}$. Therefore we obtain an identification
\begin{equation}\label{newton set for EL type} \newtonset(G) =  \{ (r_1, r_2, \cdots, r_d) \in \Q^d:  r_1 \leq r_2 \leq \cdots \leq r_d\}.\end{equation}
Under this identification, the Newton point $\newtonmap_G([b])$ of $\generalhodge{\dieudonnemodule}$ coincides with the Newton polygon of $(M_i, \phi_i)$ which was already seen to be independent of $i \in \ELgalois$. We will  refer to this polygon as the \emph{Newton polygon} of  $\generalhodge{\dieudonnemodule}$. The polygon $\newtonmap_G([b])$ is closely related with the Newton polygon of $\dieudonnemodule$ (without $\integerring$-module structure) as follows: a slope $\lambda$ appears in $\newtonmap_G([b])$ with multiplicity $\alpha$ if and only if it appears in the Newton polygon of $\dieudonnemodule$ with multiplicity $m\alpha$. 

We can also regard the $\sigma$-invariant Hodge point $\bar{\mu}$ as a polygon under the identification \eqref{newton set for EL type}. We will refer to this polygon as the \emph{$\sigma$-invariant Hodge polygon} of $\generalhodge{\dieudonnemodule}$. The inequality $\newtonmap_G([b]) \preceq \bar{\mu}$ serves as a generalized Mazur's inequality, which says that the Newton polygon $\newtonmap_G([b])$ lies above the $\sigma$-invariant Hodge polygon $\bar{\mu}$. $\generalhodge{\dieudonnemodule}$ is $\mu$-ordinary if and only if the two polygons coincide.  

When $\{\mu\}$ is minuscule, we also identify $\generalhodge{\BT} = (\BT, (t_i))$ with a $p$-divisible group $\BT$ with an action of $\integerring$. All of the discussions above evidently apply to $\generalhodge{\BT}$. Namely, we can define the type, the Newton polygon and the $\sigma$-invariant Hodge polygon of $\generalhodge{\BT}$. In addition, when $\{\mu\}$ is minuscule we have the following facts:
\begin{enumerate}
\item The $\sigma$-invariant Hodge polygon $\bar{\mu}$ of $\generalhodge{\BT}$ is determined by the type $(d, \ELf)$ as follows: if we write $\bar{\mu}=(a_1, a_2, \cdots, a_d)$, the slopes $a_j$ are given by
\[ a_j = \#\{ i \in \ELgalois : \ELf(i) > d-j\}\]
(see \cite{Moonen04}, 1.2.5.). 
\item There exists a unique isomorphism class of $\mu$-ordinary $p$-divisible groups with $\integerring$-module structure of a fixed type $(d, \ELf)$ (see \cite{Moonen04}, Theorem 1.3.7.). 
\end{enumerate}

\begin{remark} As seen in \ref{notations on Weil restriction and Frobenius actions}, we have an embedding $G_\witt = \ELgroupintegral \otimes_{\Z_p} \witt \hookrightarrow \GL(M)$ where the image is identified with a product of $m$ copies of $\GL_n \otimes_{\Z_p}\witt$. The decomposition \eqref{char space decomp EL type} shows that these copies are given by $\GL(M_i)$. In particular, we have $n=d$. 
\end{remark}
\end{example}

\subsubsection{}\label{def of deligne-lusztig set}

The isomorphism class of $\generalhodge{\dieudonnemodule}=(\dieudonnemodule, (t_i))$ depends on the choice $b \in [b]$, even though $\dieudonnemodule[1/p] \simeq N_b(\module^* \otimes_{\Z_p} \Q_p)$ is independent of this choice. To see this, let $\generalhodge{\dieudonnemodule'} = (\dieudonnemodule', (t_i'))$ be the $F$-crystal over $k$ with $G$-structure that arises from another choice $b' = g b \sigma(g)\inv \in [b] \cap G(\witt) \mu(p) G(\witt)$ for some $g \in G(\quotientwitt)$. Then $g$ gives an isomorphism 
\[\dieudonnemodule[1/p] \simeq N_b(\module^* \otimes_{\Z_p} \Q_p) \stackrel{\sim}{\longrightarrow}  N_{b'}(\module^* \otimes_{\Z_p} \Q_p) \simeq \dieudonnemodule'[1/p],\] 
which also matches $(t_i)$ with $(t_i')$ since $g \in G(\quotientwitt)$. However, this isomorphism does not induce an isomorphism between $\dieudonnemodule$ and $\dieudonnemodule'$ unless $g \in G(\witt)$.

The above discussion motivates us to consider the set
\[ \DeligneLusztig{G}{ \mu}{b} := \{ g \in G(\quotientwitt)/ G(\witt) | g b \sigma(g)\inv \in G(\witt) \mu(p) G(\witt)\}. \]
This set is clearly independent of our choice of $b \in [b]$ up to bijection. It is also independent of the choice of $\mu \in \{\mu\}$ as we alread noted that the set $G(\witt) \mu(p) G(\witt)$ only depends on the conjugacy class of $\mu$. The set $\DeligneLusztig{G}{ \mu}{b}$ is called the \emph{affine Deligne-Lusztig set} associated to $(G, [b], \{\mu\})$.

\begin{prop}\label{moduli interpretation of deligne-lusztig sets} Fix an element $b \in [b]$, and let $\generalhodge{\dieudonnemodule} = (\dieudonnemodule, (t_i))$ denote the $F$-crystal with $G$-structure induced by $b$. Then the affine Deligne-Lusztig set $\DeligneLusztig{G}{\mu}{b}$ classifies isomorphism classes of tuples $(\dieudonnemodule', (t_i'), \genqisog)$ where
\begin{itemize}
\item $(\dieudonnemodule', (t_i'))$ is an $F$-crystal over $k$ with $G$-structure;
\item $\genqisog: \dieudonnemodule'[1/p] \stackrel{\sim}{\longrightarrow} \dieudonnemodule[1/p]$ is an isomorphism which matches $(t_i')$ with $(t_i)$.
\end{itemize}
When $\{\mu\}$ is minuscule, take $\BT$ to be the $p$-divisible group with $\D(\BT) = \dieudonnemodule$. Then the set $\DeligneLusztig{G}{\mu}{b}$ also classifies isomorphism classes of tuples $(\BT', (t_i'), \genqisog)$ where
\begin{itemize}
\item $(\BT', (t_i'))$ is a $p$-divisible group over $k$ with $G$-structure;
\item $\genqisog: \BT \rightarrow \BT'$ is a quasi-isogeny such that the induced isomorphism $\D(\BT')[1/p] \stackrel{\sim}{\longrightarrow} \D(\BT)[1/p]$ matches $(t_i')$ with $(t_i)$.
\end{itemize}
\end{prop}

\begin{proof}
The second part follows immediately from the first part using Dieudonn\'e theory, so we need only prove the first part. 

Let $g$ be a representative of $g G(\witt) \in \DeligneLusztig{G}{ \mu}{ b}$. Then as discussed in \ref{def of deligne-lusztig set}, the element $b' := g\inv b \sigma(g)$ gives rise to an $F$-crystal over $k$ with $G$-structure $(\dieudonnemodule', (t_i'))$ and an isomorphism $\genqisog: \dieudonnemodule'[1/p] \stackrel{\sim}{\longrightarrow} \dieudonnemodule[1/p]$ which matches $(t_i')$ with $(t_i)$. It is clear that the isomorphism class of $(\dieudonnemodule', (t_i'), \iota)$ does not depend on the choice of the representative $g$.

Conversely, let $(\dieudonnemodule', (t_i'), \iota)$ be a tuple as in the statement. Let $b' \in G(\quotientwitt)$ be the linearization of the Frobenius map on $\dieudonnemodule'[1/p]$. Then the isomorphism $\genqisog: \dieudonnemodule'[1/p] \stackrel{\sim}{\longrightarrow} \dieudonnemodule[1/p]$ determines an element $g \in G(\quotientwitt)$ such that $b' = g b \sigma(g)\inv$. Moreover, we have $b' \in G(\witt)\mu(p)G(\witt)$ since $(\dieudonnemodule', (t_i'))$ is an $F$-crystal over $k$ with $G$-structure. Changing $(\dieudonnemodule', (t_i'), \iota)$ to an isomorphic tuple will change $g$ to $gh$ for some $h \in G(\witt)$, so we get a well-defined element $gG(\witt) \in \DeligneLusztig{G}{ \mu}{ b}$. 

These associations are clearly inverse to each other, so we complete the proof. 
\end{proof}

We now describe some functorial properties of affine Deligne-Lusztig sets which are compatible with the functorial properties of unramified local Shimura data of Hodge type described in Lemma \ref{functoriality of unram local shimura data of hodge type}.

\begin{lemma} \label{functoriality of deligne-lusztig sets} 
Let $G'$ be an unramified connected reductive group over $\Q_p$. 
\begin{enumerate} 
\item\label{product functoriality of deligne-lusztig sets} If $(G', [b'], \{\mu'\})$ is an unramified local Shimura datum of Hodge type, we have an isomorphism
\[ \DeligneLusztig{G \times G'}{ \mu , \mu'}{ b, b'} \stackrel{\sim}{\longrightarrow}\DeligneLusztig{G}{ \mu}{ b} \times \DeligneLusztig{G'}{ \mu'}{ b'}\]
induced by the natural projections. 
\item\label{hom functoriality of deligne-lusztig sets} For any homomorphism $f: G \longrightarrow G'$ defined over $\Z_p$, we have a natural map 
\[\DeligneLusztig{G}{ \mu}{ b} \longrightarrow \DeligneLusztig{G'}{ f \circ \mu}{ f(b)}\]
induced by $gG(\witt) \mapsto f(g)G'(\witt)$, which is injective if $f$ is a closed immersion. 

\end{enumerate}
\end{lemma}

\begin{proof}
The only possibly non-trivial assertion is the injectivity of the natural map $\DeligneLusztig{G}{ \mu}{ b} \longrightarrow \DeligneLusztig{G'}{ f \circ \mu}{ f(b)}$ in \eqref{hom functoriality of deligne-lusztig sets} when $f$ is a closed immersion. To see this, one may assume that $G'=\GL_n$ by embedding $G'$ into some $\GL_n$. Then the assertion follows after observing that the map
\[G(\quotientwitt)/G(\witt) \longrightarrow GL_n(\quotientwitt)/\GL_n(\witt)\]
is injective. 
\end{proof}

\subsection{Deformation Spaces of $p$-divisible groups with Tate tensors}\label{construction of universal deformation}$ $

In this subsection, we review Faltings's construction of a ``universal" deformation of $p$-divisible groups with Tate tensors, given in \cite{Faltings99}, \S7. We refer readers to \cite{Moonen98}, \S4 for a more detailed discussion of these results. 


\subsubsection{} \label{filtered dieudonne modules}

Let $\genring$ be a formally smooth $\witt$-algebra of the form $\genring = \witt[[u_1, \cdots, u_N]]$ or $\genring = \witt[[u_1, \cdots, u_N]]/(p^m)$. We can define a lift of Frobenius map on $\genring$, which we also denote by $\sigma$, by setting $\sigma(u_i) = u_i^p$. 

We define a \emph{filtered crystalline Dieudonn\'e module} over $\genring$ to be a $4$-tuple $(\deform{\dieudonnemodule}, \Fil^1(\deform{\dieudonnemodule}), \nabla, F)$ where
\begin{itemize}
\item $\deform{\dieudonnemodule}$ is a free $\genring$-module of finite rank;
\item $\Fil^1(\deform{\dieudonnemodule}) \subset \deform{\dieudonnemodule}$ is a direct summand;
\item $\nilpconnection: \deform{\dieudonnemodule} \to \deform{\dieudonnemodule} \otimes \completion{\Omega}_{\genring / \witt}$ is an integrable, topologically quasi-nilpotent connection;
\item $F_\deform{\dieudonnemodule}: \deform{\dieudonnemodule} \to \deform{\dieudonnemodule}$ is a $\sigma$-linear endomorphism,
\end{itemize} 
which satisfy the following conditions:
\begin{enumerate}[label=(\roman*)]
\item $F_\deform{\dieudonnemodule}$ induces an isomorphism $ \big(\deform{\dieudonnemodule} + p\inv \Fil^1(\deform{\dieudonnemodule})\big) \otimes_{\genring, \sigma} \genring \stackrel{\sim}{\longrightarrow} \deform{\dieudonnemodule}$, and
\item $\Fil^1(\deform{\dieudonnemodule}) \otimes_\genring (\genring/p) = \Ker \big(F \otimes \sigma_{\genring/p}: \deform{\dieudonnemodule} \otimes_\genring (\genring/p) \to  \deform{\dieudonnemodule} \otimes_\genring (\genring/p)\big)$.
\end{enumerate}
Combining the work of de Jong in \cite{deJong95} and Grothendieck-Messing theory, we obtain an equivalence between the category of filtered crystalline Dieudonn\'e modules over $\genring$ and the (opposite) category of $p$-divisible groups over $\genring$ (see also \cite{Moonen98}, 4.1.). 


\subsubsection{} \label{construction of universal dieudonne crystal}

Let $\BT$ be a $p$-divisible group over $k$. We write $\deformrings$ for the category of artinian local $\witt$-algebra with residue field $k$. By a \emph{deformation} or \emph{lifting} of $\BT$ over $\genring \in  \deformrings$, we mean a $p$-divisible group $\deform{\BT}$ over $\genring$ with an isomorphism $\alpha: \deform{\BT} \otimes_\genring k \cong \BT$. We define a functor $\Def_\BT: \deformrings \to \textbf{Sets}$ by setting $\Def_\BT(\genring)$ to be the set of isomorphism classes of deformations of $\BT$ over $\genring$. 

We write $\dieudonnemodule := \D(\BT)$ with the Frobenius map $F$, and let $\Fil^1(\dieudonnemodule) \subset \dieudonnemodule$ be its Hodge filtration. We choose a cocharacter $\mu: \Gm \to \GL_W(\dieudonnemodule)$ such that $\sigma\inv(\mu)$ induces this filtration; for instance, we take $\mu$ to be the dominant cocharacter that represents the Hodge polygon of $\BT$ under the identification of the Newton set $\newtonset(\GL_n)$ in Example \ref{newton map for GLn}. The stabilizer of the complement of $\Fil^1(\dieudonnemodule)$ is a parabolic subgroup. We let $\unipotent^\mu$ be its unipotent radical, and take the formal completion $\completion{\unipotent}^\mu = \Spf \genring_\GL^\mu$ of $\unipotent^\mu$ at the identity section. Then $\genring_\GL^\mu$ is a formal power series ring over $\witt$, so we can define a lift of Frobenius map on $\genring_\GL^\mu$.

\begin{prop}[\cite{Faltings99}, \S7] \label{universal deformation of a p-divisible group} Let $u_t \in \completion{\unipotent}^\mu(\genring_\GL^\mu)$ be the tautological point. Define
\[ \deform{\dieudonnemodule} : = \dieudonnemodule \otimes_\witt \genring_\GL^\mu, \quad \Fil^1(\deform{\dieudonnemodule}):= \Fil^1(\dieudonnemodule) \otimes_\witt \genring_\GL^\mu, \quad F_\deform{\dieudonnemodule} := u_t \circ (F \otimes_\witt \sigma).\]
\begin{enumerate}
\item\label{existence of horizontal connection} There exists a unique topologically quasi-nilpotent connection $\nilpconnection :\deform{\dieudonnemodule} \to \deform{\dieudonnemodule} \otimes \completion{\Omega}_{\genring_\GL^\mu / \witt}$ that commutes with $F_\deform{\dieudonnemodule}$, and this connection is integrable.
\item\label{description of universal deformation in terms of filtered dieudonne module} If $p >2$, the filtered crystalline Dieudonn\'e module $(\deform{\dieudonnemodule},  \Fil^1(\deform{\dieudonnemodule}), \nilpconnection, F_\deform{\dieudonnemodule})$ corresponds to the universal deformation of $\BT$ via the equivalence described in \ref{filtered dieudonne modules}.
\end{enumerate}
\end{prop}
In particular, \eqref{description of universal deformation in terms of filtered dieudonne module} implies that we have an identification $\Def_\BT \cong \Spf \genring_\GL^\mu$. We will write $\deform{\BT}_\GL^\mu$ for the universal deformation of $\BT$.

\subsubsection{} We now consider deformations of $p$-divisible groups with $G$-structure. We fix an unramified local Shimura datum of Hodge type $(G, [b], \{\mu\})$ with minuscule $\{\mu\}$. We also fix a faithful $G$-representation $\module \in \text{Rep}_{\Z_p}(G)$ in the condition \ref{unram loc shimura datum condition 2} of \ref{definition of local shimura datum of hodge type}, and choose $b \in [b]$ and $\mu \in \{\mu\}$ such that $b \in G(\witt)\mu(p)G(\witt)$. Then we obtain an $F$-crystal with $G$-structure $\generalhodge{\dieudonnemodule}=(\dieudonnemodule, (t_i))$ as explained in \ref{description of F-crystals with tensors}, which gives rise to a $p$-divisible group with $G$-structure $\generalhodge{\BT} = (\BT, (t_i))$ since $\{\mu\}$ is minuscule. The condition $b \in G(\witt)\mu(p)G(\witt)$ assures that the Hodge filtration $\Fil^1(\dieudonnemodule) \subset \dieudonnemodule$ is induced by $\sigma\inv(\mu)$, so all the constructions from \ref{construction of universal dieudonne crystal} and Proposition \ref{universal deformation of a p-divisible group} are valid for $\BT$.

Let $\unipotent_G^\mu := \unipotent^\mu \cap G_\witt$, which is a smooth unipotent subgroup of $G_W$. Take $\completion{\unipotent}_G^\mu = \Spf \genring_G^\mu$ to be its formal completion at the identity section. Then $\genring_G^\mu$ is a formal power series ring over $\witt$, so we get a lift of Frobenius map to $\genring_G^\mu$. Alternatively, we get this lift from the lift on $\genring_\GL^\mu$ via the surjection $\genring_\GL^\mu \twoheadrightarrow \genring_G^\mu$ induced by the embedding $\completion{\unipotent}_G^\mu \hooklongrightarrow \completion{\unipotent}^\mu$.

Let $u_{t, G} \in \completion{\unipotent}_G^\mu(\genring_G^\mu)$ be the tautological point. Define
\[ \deform{\dieudonnemodule}_G : = \dieudonnemodule \otimes_\witt \genring_G^\mu, \quad \Fil^1(\deform{\dieudonnemodule}_G):= \Fil^1(\dieudonnemodule) \otimes_\witt \genring_G^\mu, \quad F_{\deform{\dieudonnemodule}_G} := u_{t, G} \circ (F \otimes_\witt \sigma).\]
Then we have an integrable, topologically quasi-nilpotent connection $\nilpconnection_G :\deform{\dieudonnemodule}_G \to \deform{\dieudonnemodule}_G \otimes \completion{\Omega}_{\genring_G^\mu / \witt}$ induced by $\nilpconnection :\deform{\dieudonnemodule} \to \deform{\dieudonnemodule} \otimes \completion{\Omega}_{\genring_\GL^\mu / \witt}$ from Proposition \ref{universal deformation of a p-divisible group}. In addition, $\nilpconnection_G$ clearly commutes with $F_{\deform{\dieudonnemodule}_G}$ by construction. Hence we have a filtered crystalline Dieudonn\'e module $(\deform{\dieudonnemodule}_G,  \Fil^1(\deform{\dieudonnemodule}_G), \nilpconnection_G, F_{\deform{\dieudonnemodule}_G})$. 

Note that $ \deform{\dieudonnemodule}_G$ is equipped with tensors $(\liftlow{t}^\univ_i) := ( t_i \otimes 1)$, which are evidently $F_{\deform{\dieudonnemodule}_G}$-invariant by construction. If $p>2$, one can prove that these tensors lie in the 0th filtration (see \cite{Kim13}, Lemma 2.2.7 and Proposition 2.5.9.). 

Let $\deform{\BT}_G^\mu$ be the $p$-divisible group over $\genring_G^\mu$ corresponding to $(\deform{\dieudonnemodule}_G,  \Fil^1(\deform{\dieudonnemodule}_G), \nabla_G, F_{\deform{\dieudonnemodule}_G})$ via the equivalence described in \ref{filtered dieudonne modules}. Alternatively, one can get $\deform{\BT}_G^\mu$ by simply pulling back $\deform{\BT}_\GL^\mu$ over $\genring_G^\mu$. Then $\deform{\BT}_G^\mu$ is the ``universal deformation" of $(\BT, (t_i))$ in the following sense:

\begin{prop}[\cite{Faltings99}, \S7] \label{universal deformation of a p-divisible group with tate tensors} Assume that $p>2$. Let $\genring$ be a formally smooth $\witt$-algebra of the form $\genring = \witt[[u_1, \cdots, u_N]]$ or $\genring = \witt[[u_1, \cdots, u_N]]/(p^m)$. Choose a deformation $\deform{\BT}$ of $\BT$ over $\genring$, and let $f: \genring_\GL^\mu \to \genring$ be the morphism induced by $\deform{\BT}$ via $\Spf \genring_\GL^\mu \cong \Def_\BT$. Then $f$ factors through $\genring_G^\mu$ if and only if the tensors $(t_i)$ can be lifted to tensors $(\liftlow{t}_i) \in \D(\deform{\BT})^\otimes$ which are Frobenius-invariant and lie in the $0$th filtration with respect to the Hodge filtration. If this holds, then we necessarily have $(f^* \liftlow{t}^\univ_i) = (\liftlow{t}_i)$. 
\end{prop}

We define $\Def_{\BT, G}$ to be the image of the closed immersion $\Spf \genring_G^\mu \hooklongrightarrow \Spf \genring_\GL^\mu \cong \Def_\BT$. Then $\Def_{\BT, G}$ classifies deformations of $(\BT, (t_i))$ over formal power series rings over $\witt$ or $\witt/(p^m)$ in the sense of Proposition \ref{universal deformation of a p-divisible group with tate tensors}. Note that our definition of $\Def_{\BT, G}$ is independent of the choice of $(t_i)$ and $\mu \in \{\mu\}$; indeed, the independence of the choice of $(t_i)$ is clear by construction, and the independence of the choice of $\mu$ follows from the universal property.

We close this section with some functorial properties of deformation spaces, which are compatible with the functorial properties of unramified local Shimura data of Hodge type described in Lemma \ref{functoriality of unram local shimura data of hodge type}. The proof is straightforward and thus omitted.

\begin{lemma} \label{functoriality of deformation spaces}
Let $(G', [b'], \{\mu'\})$ be another unramified local Shimura datum of Hodge type. Choose $b' \in [b']$ and $\mu' \in \{\mu'\}$ such that $b' \in G'(\witt) \mu'(p) G'(\witt)$, and let $(\BT', (t_i'))$ be a $p$-divisible group with $G'$-structure that arises from this choice. 
\begin{enumerate}
\item\label{product functoriality of deformation spaces} The natural morphism $\Def_\BT \times \Def_{\BT'} \longrightarrow \Def_{\BT \times \BT'}$, defined by taking the product of deformations, induces an isomorphism
\[ \Def_{\BT, G} \times \Def_{\BT', G'} \stackrel{\sim}{\longrightarrow} \Def_{\BT \times \BT', G \times G'}.\]
\item\label{hom functoriality of deformation spaces} For any homomorphism $f : G \to G'$ defined over $\Z_p$ such that $f(b) = b'$, we have a natural morphism 
\[\Def_{\BT, G} \to \Def_{\BT', G'}\]
induced by the map $\completion{\unipotent}_G^\mu \to \completion{\unipotent}_{G'}^{f \circ \mu}$. 
\end{enumerate}
\end{lemma}

\begin{remark}
With some additional work, one can show that the natural morphism $\Def_{\BT, G} \to \Def_{\BT', G'}$ in \eqref{hom functoriality of deformation spaces} is independent of the choice of $\mu \in \{\mu\}$. See \cite{Kim13}, Proposition 3.7.2 for details. 
\end{remark}

\section{Hodge-Newton reducible local Shimura data of Hodge type}\label{Hodge-Newton filtration for Hodge type}

In this section, we state and prove our main results on the Hodge-Newton decomposition and the Hodge-Newton filtration in the setting of unramified local Shimura data of Hodge type.

\subsection{EL realization of Hodge-Newton reducibility}$ $

\subsubsection{}\label{definition of hodge-newton type} Let $(G, [b], \{\mu\})$ be an unramified local Shimura datum of Hodge type. Choose a maximal torus $T \subseteq G$ and a Borel subgroup $\borel \subseteq G$ containing $T$, both defined over $\Z_p$. Let $\parabolic$ be a proper standard parabolic subgroup of $G$ with Levi factor $\levi$ and unipotent radical $\unipotent$. We say that $(G, [b], \{\mu\})$ is \emph{Hodge-Newton reducible} (with respect to $\parabolic$ and $\levi$) if there exist $\mu \in \{\mu\}$ which factors through $\levi$ and an element $b \in [b] \cap \levi(\quotientwitt)$ which satisfy the following conditions:
\begin{enumerate}[label=(\roman*)]
\item\label{hodge-newton type levi condition} $[b]_\levi \in B(\levi, \{\mu\}_\levi)$,
\item\label{hodge-newton type cocharacters condition} in the action of $\mu$ and $\newtonmap_G([b])$ on $\Lie (\unipotent) \otimes_{\Q_p} \quotientwitt$, only non-negative characters occur.
\end{enumerate}
Since $G$ is unramified, one can give an alternative definition in terms of some specific choice of $b \in [b] \cap \levi(\quotientwitt)$ and $\mu \in \{\mu\}$ (see \cite{Rapoport-Viehmann14}, Remark 4.25.).

\begin{example}\label{hodge-newton type for EL type}

Consider the case $G = \ELgroupintegral$ where $\integerring$ is the ring of integers of some finite unramified extension of $\Q_p$. Then $\levi$ is of the form 
\[\levi = \ELgroupintegral[j_1] \times \ELgroupintegral[j_2] \times \cdots \times \ELgroupintegral[j_r]. \]
Recall from Example \ref{BT with O-module structure} that we have an identification
\[\newtonset(G) =  \{ (r_1, r_2, \cdots, r_n) \in \Q^d:  r_1 \leq r_2 \leq \cdots \leq r_n\}.\]
Using this, we may write $\newtonmap_G([b]) = (\newtonmap_1, \newtonmap_2, \cdots, \newtonmap_n)$ and $\bar{\mu} = (\mu_1, \mu_2, \cdots, \mu_r)$. Then $(G, [b], \{\mu\})$ is of Hodge-Newton reducible with respect to $\parabolic$ and $\levi$ if and only if the following conditions are satisfied for each $k=1, 2, \cdots, r$:
\begin{enumerate}[label=(\roman*')]
\item $\newtonmap_1 + \newtonmap_2 + \cdots + \newtonmap_{j_k} = \mu_1 + \mu_2 + \cdots + \mu_{j_k}$,
\item $\newtonmap_{j_k}<\newtonmap_{j_k+1}$.
\end{enumerate}
In other words, $(G, [b], \{\mu\})$ is of Hodge-Newton reducible (with respect to $\parabolic$ and $\levi$) if and only if the Newton polygon $\newtonmap_G([b])$ and the $\sigma$-invariant Hodge polygon $\bar{\mu}$ have contact points which are break points of $\newtonmap_G([b])$ specified by $\levi$. We refer the readers to \cite{Mantovan-Viehmann10}, \S3 for more details. 
\end{example}

\subsubsection{}

For the rest of this section, we fix an unramified local Shimura datum of Hodge type $(G, [b], \{\mu\})$ which is Hodge-Newton reducible with respect to $\parabolic$ and $\levi$. Let us also fix a faithful $G$-representation $\module \in \text{Rep}_{\Z_p}(G)$ in the condition \ref{unram loc shimura datum condition 2} of \ref{definition of local shimura datum of hodge type}.  Our strategy is to study $(G, [b], \{\mu\})$ by embedding $G$ into another group $\EL{G}$ of EL type such that the datum $(\EL{G}, [b], \{\mu\})$ is also Hodge-Newton reducible. 

Note that if $G$ is not split, the datum $(\GL(\module), [b], \{\mu\})$ is not Hodge-Newton reducible in general. In fact, the map on the Newton sets 
$\newtonset(G) \longrightarrow \newtonset(\GL(\module))$ 
induced by the embedding $G \hooklongrightarrow \GL(\module)$ does not map $\bar{\mu}_G$ to the Hodge polygon $\mu_{\GL(\module)}$ since it does not respect the action of $\sigma$.

\begin{lemma}\label{EL realization of hodge-newton type}
There exists a group $\EL{G}$ of EL type with the following properties:
\begin{enumerate}[label=(\roman*)]
\item the embedding $G \hookrightarrow \GL(\module)$ factors through $\EL{G}$.
\item the datum $(\EL{G}, [b], \{\mu\})$ is Hodge-Newton reducible with respect to a parabolic subgroup $\EL{\parabolic} \subsetneq \EL{G}$ and its Levi factor $\EL{\levi}$ such that $\parabolic = \EL{\parabolic} \cap G$ and $\levi = \EL{\levi} \cap G$. 
\end{enumerate}
\end{lemma}

\begin{proof}
Write $V := \module \otimes_{\Z_p} {\Q^\unram}$ where $\Q^\unram$ is the maximal unramified extension of $\Q_p$ in a fixed algebraic closure. We know that $G$ is split over $\Q^\unram$ for being unramified over $\Q_p$. Hence $V$ admits a decomposition into character spaces
\begin{equation}\label{char space decomp over K0} V = \bigoplus_{\chi \in \characters(T)} V_\chi \end{equation}
with the property that $\sigma(V_\chi) = V_{\sigma\chi}.$

For each $\chi \in \characters(T)$, let $\weylorbit{\chi}$ denote the $\weylgroup$-conjugacy class of $\chi$ and write $V_{\weylorbit{\chi}} := \oplus_{\omega \in \weylgroup} V_{\omega \cdot \chi}$. Since $V$ is a $G$-representation, we can rewrite the decomposition \eqref{char space decomp over K0} as
\[ V = \bigoplus_{\weylorbit{\chi} \in \characters(T)/\weylgroup} V_{\weylorbit{\chi}}\]
where $V_{\weylorbit{\chi}}$'s are sub $G$-representations (see \cite{Serre68}, Theorem 4.) with the property that $V_{\weylorbit{\sigma \chi}} = \sigma(V_{\weylorbit{\chi}})$. If a $\weylgroup$-conjugacy class $\weylorbit{\chi} \in \characters(T)/\weylgroup$ is in an orbit of size $m$ under the action of $\sigma$, the $G$-representation
\[ \bigoplus_{i=0}^{m-1} V_{\weylorbit{\sigma^i \chi}}\]
is also a $\ELgroup$-representation where $\unramext$ is the field of definition of $\weylorbit{\chi}$, which is a degree $m$  unramified extension of $\Q_p$ (cf. \eqref{char space decomp EL type} in Example \ref{BT with O-module structure}). Hence the embedding $G_{\Q_p} \hookrightarrow \GL(\module_{\Q_p})$ factors through a group of the form $\prod \ELgroup[j]$ where each $\unramext_j$ is the field of definition of an orbit in $\characters(T)/\weylgroup$.  Then by \cite{Serre68}, Theorem 5, we can take the pull-back of this embedding over $\Z_p$ to obtain
\[G \hooklongrightarrow \prod \Res_{\integerring_j|\Z_p} \GL_{n_j} \hooklongrightarrow \GL(\module) \]
where $\integerring_j$ is the ring of integers of $\unramext_j$. 

We take 
\[\EL{G}:= \prod \Res_{\integerring_j|\Z_p} \GL_{n_j}.\]
Choose a Borel pair $(\EL{\borel}, \EL{T})$ of $\EL{G}$ such that $\borel \subseteq \EL{\borel}$ and $T \subseteq \EL{T}$. Then we get a proper standard parabolic subgroup $\EL{\parabolic} \subsetneq \EL{G}$ with Levi factor $\EL{\levi}$ such that $\parabolic = \EL{\parabolic} \cap G$ and $\levi = \EL{\levi} \cap G$ (e.g. by using \cite{SGA3}, Exp. XXVI, Cor. 6.10.). 

It is evident from the construction that the embedding $G \hooklongrightarrow \EL{G}$ respects the action of $\sigma$ on cocharacters. Hence the induced map on the Newton sets $\newtonset(G) \longrightarrow \newtonset(\EL{G})$ maps $\bar{\mu}_G$ to the $\sigma$-invariant Hodge polygon $\bar{\mu}_{\EL{G}}$. Combining this fact with the functoriality of the Kottwitz map and the Newton map, we verify that the datum $(\EL{G}, [b], \{\mu\})$ is Hodge-Newton reducible with respect to $\EL{\parabolic}$ and $\EL{\levi}$. 
\end{proof}

We will refer to the datum $(\EL{G}, [b], \{\mu\})$ in Lemma \ref{EL realization of hodge-newton type} as an \emph{EL realization} of the Hodge-Newton reducible datum $(G, [b], \{\mu\})$. 

\begin{remark}
If $G$ is split, the construction in the proof above yields $\EL{G} = \GL(\module)$. 
\end{remark}

\subsection{The Hodge-Newton decomposition and the Hodge-Newton filtration}$ $

\subsubsection{}\label{construction of sub local shimura data for hodge newton reducibility}
Fix an EL realization $(\EL{G}, [b], \{\mu\})$ of our datum $(G, [b], \{\mu\})$, and take $\EL{\parabolic}$ and $\EL{\levi}$ as in Lemma \ref{EL realization of hodge-newton type}. In a view of the functorial properties in Lemma \ref{functoriality of unram local shimura data of hodge type}, Lemma \ref{functoriality of deligne-lusztig sets} and Lemma \ref{functoriality of deformation spaces}, we will always assume for simplicity that $\EL{G}$ is of the form
\[\EL{G}:= \ELgroupintegral\] 
where $\integerring$ is the ring of integers of some finite unramified extension $\unramext$ of $\Q_p$. Then $\EL{\levi}$ is of the form 
\begin{equation}\label{EL levi decomp}\EL{\levi} = \ELgroupintegral[j_1] \times \ELgroupintegral[j_2] \times \cdots \times \ELgroupintegral[j_r]. \end{equation}

Let us now choose $b \in [b] \cap \levi(\quotientwitt)$ and $\mu \in \{\mu\}$ as in \ref{hodge-newton type levi condition} of \ref{definition of hodge-newton type}. After taking $\sigma$-conjugate in $\levi(\quotientwitt)$ if necessary, we may assume that $b \in \levi(\witt) \mu(p) \levi(\witt)$. Let $\generalhodge{\dieudonnemodule} = (\dieudonnemodule, (t_i))$ be the corresponding $F$-crystal over $k$ with $G$-structure. If $\{\mu\}$ is minuscule, we let $\generalhodge{\BT} = (\BT, (t_i))$ denote the corresponding $p$-divisible group over $k$ with $G$-structure. 

Note that the tuple $(\levi, [b]_\levi, \{\mu\}_\levi)$ is an unramified local Shimura datum of Hodge type; indeed, with our choice of $b \in [b]_\levi$ and $\mu \in \{\mu\}_\levi$ one immediately verifies the conditions \ref{unram loc shimura datum condition 1'} and \ref{unram loc shimura datum condition 2'} of \ref{definition of local shimura datum of hodge type}.

\begin{theorem}\label{existence of hodge newton decomp in char p} Notations as above. In addition, we set the following notations:
\begin{itemize}
\item $\EL{\levi}_j$ denotes the $j$-th factor in \eqref{EL levi decomp},
\item $\levi_j$ is the image of $\levi$ under the projection $\EL{\levi} \twoheadrightarrow \EL{\levi}_j$,
\item $b_j$ is the image of $b$ under the projection $\levi \twoheadrightarrow \levi_j$,
\item $\mu_j$ is the cocharacter of $\levi_j$ obtained by composing $\mu$ with the projection $\levi \twoheadrightarrow \levi_j$.
\end{itemize}
Then $\generalhodge{\dieudonnemodule}$ admits a decomposition
\begin{equation}\label{hodge-newton decomp f-crystal} \generalhodge{\dieudonnemodule} = \generalhodge{\dieudonnemodule}_1 \times \generalhodge{\dieudonnemodule}_2 \times \cdots \times \generalhodge{\dieudonnemodule}_r\end{equation}
where $\generalhodge{\dieudonnemodule}_j$ is an $F$-crystal with $\levi_j$-structure that arises from an unramified local Shimura datum of Hodge type $(\levi_j, [b_j], \{\mu_j\})$. 

When $\{\mu\}$ is minuscule, we also have a decomposition
\begin{equation}\label{hodge-newton decomp p-divisible group}\generalhodge{\BT} = \generalhodge{\BT}_1 \times  \generalhodge{\BT}_2 \times \cdots \times \generalhodge{\BT}_r\end{equation}
where $\generalhodge{\BT}_j$ is a $p$-divisible group with $\levi_j$-structure corresponding to $\generalhodge{\dieudonnemodule}_j$. 
\end{theorem}

\begin{proof}
We need only prove the first part, as the second part follows immediately from the first part via Dieudonn\'e theory. 

Considering $b$ as an element of $[b]_{\EL{G}}$, we get an $F$-crystal over $k$ with $\EL{G}$-structure $\EL{\dieudonnemodule}$ from an unramified local Shimura datum of Hodge type $(\EL{G}, [b], \{\mu\})$. As explained in Example \ref{BT with O-module structure}, we can regard $\EL{G}$-structure as an action of $\integerring$ which we refer to as $\integerring$-module structure. Since $(\EL{G}, [b], \{\mu\})$ is Hodge-Newton reducible, \cite{Mantovan-Viehmann10}, Corollary 7 yields a decomposition
\begin{equation}\label{EL hodge-newton decomp} \EL{\dieudonnemodule} = \EL{\dieudonnemodule}_1 \times \EL{\dieudonnemodule}_2 \times \cdots \times \EL{\dieudonnemodule}_r\end{equation}
where $\EL{\dieudonnemodule}_j$ is an $F$-crystal over $k$ with $\integerring$-module structure which arises from an unramified local Shimura datum of Hodge type $(\EL{\levi}_j, [b_j], \{\mu_j\})$. In fact, $\EL{\dieudonnemodule}_j$ corresponds to the choice $b_j \in [b_j]$ (and $\mu_j \in \{\mu_j\}$).

\emph{A priori}, it is not clear that the tuple $(\EL{\levi}_j, [b_j], \{\mu_j\})$ is an unramified local Shimura datum of Hodge type. This is indeed implicitly implied in the statement and the proof of \cite{Mantovan-Viehmann10}, Corollary 7.

We check that the tuple $(\levi_j, [b_j], \{\mu_j\})$ is an unramified local Shimura datum of Hodge type by verifying the conditions \ref{unram loc shimura datum condition 1'} and \ref{unram loc shimura datum condition 2'} of \ref{definition of local shimura datum of hodge type}. For \ref{unram loc shimura datum condition 1'}, we simply observe that $b_j \in \levi_j(\witt)\mu(p)\levi_j(\witt)$, which follows from our assumption that $b \in \levi(\witt)\mu(p)\levi(\witt)$ using the decomposition $\EL{\levi} = \EL{\levi}_1 \times \EL{\levi}_2 \times \cdots \times \EL{\levi}_r$. Then the condition \ref{unram loc shimura datum condition 2'} immediately follows since we already know that $\dieudonnemodule_j$ gives the desired $\witt$-lattice for $b_j$ and $\mu_j$.

Since $(\levi_j, [b_j], \{\mu_j\})$ is an unramified local Shimura datum of Hodge type, we can equip each $\dieudonnemodule_j$ with $\levi_j$-structure corresponding to the choice $b_j \in [b_j]$ (and $\mu_j \in \{\mu_j\}$). We thus get the desired decomposition \eqref{hodge-newton decomp f-crystal} from the decomposition \eqref{EL hodge-newton decomp}. 
\end{proof}

\begin{remark} We give an alternative proof of Theorem \ref{existence of hodge newton decomp in char p} using affine Deligne-Lusztig sets. After proving that the tuples $(\levi_j, [b_j], \{\mu_j\})$ are unramified local Shimura data of Hodge type, we find the following maps of affine Deligne-Lusztig sets:
\[ \DeligneLusztig{G}{\mu}{b} \stackrel{\sim}{\longrightarrow} \DeligneLusztig{\levi}{\mu}{b} \hookrightarrow \DeligneLusztig{\levi_1}{ \mu_1}{ b_1} \times \DeligneLusztig{\levi_2}{ \mu_2}{ b_2} \times \cdots \times \DeligneLusztig{\levi_s}{ \mu_r}{ b_r}.   \]
Here the first isomorphism is given by \cite{Mantovan-Viehmann10}, Theorem 6, whereas the second map is induced by the embedding $\levi \hookrightarrow \levi_1 \times \levi_2 \times \cdots \times \levi_j$ as in Lemma \ref{functoriality of deligne-lusztig sets}. Now the desired decomposition follows from the composition of these two maps via the moduli interpretation of affine Deligne-Lusztig sets given in Proposition \ref{moduli interpretation of deligne-lusztig sets}.
\end{remark}

\subsubsection{}\label{definition of hodge newton filtration}

We will refer to the decomposition \eqref{hodge-newton decomp f-crystal} in Theorem \ref{existence of hodge newton decomp in char p} as the \emph{Hodge-Newton decomposition} of $\generalhodge{\dieudonnemodule}$ (associated to $\parabolic$ and $\levi$). For $1 \leq a \leq b \leq r$, we define
\begin{equation*} \dieudonnemodule_{a, b} := \prod_{s=a}^b \dieudonnemodule_s.\end{equation*}
Then we obtain a filtration
\begin{equation}\label{hodge-newton filtration f-crystal} 0 \subset \dieudonnemodule_{1, 1} \subset \dieudonnemodule_{1, 2} \subset \cdots \subset \dieudonnemodule_{1, r} = \dieudonnemodule\end{equation}
such that each quotient $\dieudonnemodule_{1, s}/\dieudonnemodule_{1, s-1} \simeq \dieudonnemodule_s$ carries $\levi_s$-structure. We call this filtration the \emph{Hodge-Newton filtration} of $\generalhodge{\dieudonnemodule}$ (associated to $\parabolic$ and $\levi$). 

When $\{\mu\}$ is minuscule, we will refer to the decomposition \eqref{hodge-newton decomp p-divisible group} in Theorem \ref{existence of hodge newton decomp in char p} as the \emph{Hodge-Newton decomposition} of $\generalhodge{\BT}$ (associated to $\parabolic$ and $\levi$). For $1 \leq a \leq b \leq r$, we define
\begin{equation*} \BT_{a, b} := \prod_{s=a}^b \BT_s.\end{equation*}
Then via (contravariant) Dieudonn\'e theory, the filtration \eqref{hodge-newton filtration f-crystal} yields a filtration
\begin{equation}\label{hodge-newton filtration p-divisible group} 0 \subset \BT_{r, r} \subset \BT_{r-1, r} \subset \cdots \subset \BT_{1, r} = \BT,\end{equation}
where each quotient $\BT_{s, r}/\BT_{s+1, r} \simeq \BT_s$ carries $\levi_s$-structure. We call this filtration the \emph{Hodge-Newton filtration} of $\generalhodge{\BT}$ (associated to $\parabolic$ and $\levi$).

\begin{theorem}\label{lifting of hodge newton filtration} Assume that $p>2$ and $\{\mu\}$ is minuscule. Let $\genring$ be a formally smooth $\witt$-algebra of the form $\genring = \witt[[u_1, \cdots, u_N]]$ or $\genring = \witt[[u_1, \cdots, u_N]]/(p^m)$. Let $\deform{\generalhodge{\BT}}$ be a deformation of $\generalhodge{\BT}$ over $\genring$ with an isomorphism $\alpha: \deform{\generalhodge{\BT}} \otimes_\genring k \cong \generalhodge{\BT}$. Then there exists a unique filtration of $\deform{\BT}$
\[ 0 \subset \deform{\BT}_{r, r} \subset \deform{\BT}_{r-1, r} \subset \cdots \subset \deform{\BT}_{1, r} = \deform{\BT} \]
which lifts the Hodge-Newton filtration \eqref{hodge-newton filtration p-divisible group} in the sense that $\alpha$ induces isomorphisms $\deform{\BT}_{s, r} \otimes_\genring k \cong \BT_{s,r}$ and $\generalhodge{\deform{\BT}_{s, r}/\deform{\BT}_{s+1, r}} \otimes_\genring k \cong \generalhodge{\BT}_s$ for $s = 1, 2, \cdots, r$. 
\end{theorem}

Note that we require each quotient $\generalhodge{\deform{\BT}_{s, r}/\deform{\BT}_{s+1, r}}$ to lift tensors on $\generalhodge{\BT}_s$. 

\begin{proof} 
We will only consider the case $r=2$ as the argument easily extends to the general case. 

Take unramified local Shimura data of Hodge type $(\levi_j, [b_j], \{\mu\}_j)$ and $(\EL{\levi}_j, [b_j], \{\mu_j\})$ as in Theorem \ref{existence of hodge newton decomp in char p}. In addition, let $\EL{\BT}$ be the $p$-divisible group over $k$ with $\integerring$-module structure that arises from the datum $(\EL{G}, [b], \{\mu\})$ with the choice $b \in [b]$, and let $\EL{\BT}_j$ be the $p$-divisible group over $k$ with $\integerring$-module structure that arises from the datum $(\EL{\levi}_j, [b_j], \{\mu_j\})$ with the choice $b_j \in [b_j]$. Then the filtration
\[ 0 \subseteq \EL{\BT}_2 \subseteq \EL{\BT}\]
is the Hodge-Newton filtration of $\EL{\BT}$.

By the functorial properties of deformation spaces in Lemma \ref{functoriality of deformation spaces}, the closed embedding $G \hooklongrightarrow \EL{G}$ induces a closed embedding
\[ \Def_{\BT, G} \hooklongrightarrow \Def_{\BT, \EL{G}}.\]
Thus $\deform{\generalhodge{\BT}}$ yields a deformation $\deform{\EL{\BT}}$ of $\EL{\BT}$ over $\genring$. Then by \cite{Shen13}, Theorem 5.4, $\deform{\EL{\BT}}$ admits a (unique) filtration
\[ 0 \subseteq \deform{\BT}_2 \subseteq \deform{\BT}\]
such that $\alpha$ induces isomorphisms $\alpha_1: \EL{\deform{\BT}/\deform{\BT}_2} \otimes_\genring k \cong \EL{\BT}_1$ and $\alpha_2: \deform{\EL{\BT}}_2 \otimes_\genring k \cong \EL{\BT}_2$.

It remains to show that $\deform{\BT}/\deform{\BT}_2$ and $\deform{\BT}_2$ are equipped with tensors which lift the tensors of $\generalhodge{\BT}_1$ and $\generalhodge{\BT}_2$ respectively in the sense of Proposition \ref{universal deformation of a p-divisible group with tate tensors}. 
Note that we have isomorphisms of Dieudonn\'e modules
\[ \beta: \D(\deform{\BT} \otimes_\genring k) \cong \D(\BT), \quad \beta_1: \D((\deform{\BT}/\deform{\BT}_2) \otimes_\genring k) \cong \D(\BT_1), \quad \beta_2: \D(\deform{\BT}_2 \otimes_\genring k) \cong \D(\BT_2)\]
corresponding to the isomorphisms $\alpha, \alpha_1$ and $\alpha_2$. We may regard $\beta$ as an element of $G(\witt)$ by identifying both modules with $\module^* \otimes_{\Z_p} \witt$. Similarly, we may regard each $\beta_j$ as an element of $\EL{\levi}_j(\witt)$. Then $\beta_j$ should be in the image of $\EL{\levi}(\witt) \cap G(\witt) = \levi(\witt)$ under the projection $\EL{\levi} \twoheadrightarrow \EL{\levi}_2$ since it is induced by $\beta$. Hence we have $\beta_j \in \levi_j(\witt)$ for each $j=1, 2$.  This implies that $\deform{\BT}/\deform{\BT}_2$ and $\deform{\BT}_2$ respectively lift the tensors of $\generalhodge{\BT}_1$ and $\generalhodge{\BT}_2$ via $\alpha_1$ and $\alpha_2$, completing the proof. 
\end{proof}

\section{Serre-Tate theory for local Shimura data of Hodge type} \label{serre tate theory for hodge type}

Our goal for this section is to establish a generalization of Serre-Tate deformation theory for $p$-divisible groups that arise from $\mu$-ordinary local Shimura data of Hodge type. There are two main ingredients for our theory, namely
\begin{enumerate}[label=(\alph*)]
\item\label{serre-tate ingredient slope filtration} existence of a ``slope filtration'' which admits a unique lifting over deformation rings;
\item\label{serre-tate ingredient canonical deformation} existence of  a ``canonical deformation''. 
\end{enumerate}
We prove \ref{serre-tate ingredient slope filtration} by applying Theorem \ref{existence of hodge newton decomp in char p} and Theorem \ref{lifting of hodge newton filtration} to $\mu$-ordinary local Shimura data of Hodge type. To prove \ref{serre-tate ingredient canonical deformation}, we first embed our deformation space into a deformation space that arises from an EL realization of our local Shimura datum (cf. the proof of Theorem \ref{lifting of hodge newton filtration}), then use the existence of a canonical deformation in the latter space proved by Moonen in \cite{Moonen04}. 

Throughout this section, we will assume that $p>2$.

\subsection{The slope filtration of $\mu$-ordinary $p$-divisible groups}$ $

\subsubsection{}\label{notations for serre-tate theory}
Let us first fix some notations for this section. We fix a $\mu$-ordinary unramified local Shimura datum of Hodge type $(G, [b], \{\mu\})$. We assume that $\{\mu\}$ is minuscule, and take a unique dominant representative $\mu \in \{\mu\}$. Then we have $[b] = [\mu(p)]$ by definition of $\mu$-ordinariness, so we may take $b=\mu(p)$ and write $\generalhodge{\BT}$ for the $p$-divisible group over $k$ with $G$-structure that arises from this choice $b \in [b] \cap G(\witt)\mu(p)G(\witt)$. 
Let $m$ be a positive integer such that $\sigma^m(\mu) = \mu$, and take $\levi$ to be the centralizer of $m \cdot \bar{\mu}$ in $G$ which is a Levi subgroup (see \cite{SGA3}, Exp. XXVI, Cor. 6.10.). We set $\parabolic$ to be a proper standard parabolic subgroup of $G$ with Levi factor $\levi$.

\subsubsection{}\label{slope decomposition and slope filtration for hodge type}
One can check that $(G, [b], \{\mu\})$ is Hodge-Newton reducible with respect to $\parabolic$ and $\levi$ (see \cite{Wortmann13}, Proposition 7.4.). Hence Theorem \ref{existence of hodge newton decomp in char p} gives us the Hodge-Newton decomposition associated to $\parabolic$ and $\levi$
\begin{equation}\label{slope decomposition for hodge type} \generalhodge{\BT} = \generalhodge{\BT}_1 \times \generalhodge{\BT}_2 \times \cdots \times \generalhodge{\BT}_r \end{equation}
which we call the \emph{slope decomposition} of $\generalhodge{\BT}$. If we set
\begin{equation*} \BT_{a, b} := \prod_{s=a}^b \BT_s\end{equation*}
for $1 \leq a \leq b \leq r$, we obtain the induced Hodge-Newton filtration
\begin{equation}\label{slope filtration for hodge type} 0 \subset \BT_{r, r} \subset \BT_{r-1, r} \subset \cdots \subset \BT_{1, r} = \BT,\end{equation}
which we refer to as the \emph{slope filtration} of $\generalhodge{\BT}$. 

Now Theorem \ref{lifting of hodge newton filtration} readily gives us the first main ingredient of the theory, namely the unique lifting of the slope filtration.

\begin{prop}\label{lifting of the slope filtration for hodge type} Let $\genring$ be a $\witt$-algebra of the form $\genring = \witt[[u_1, \cdots, u_N]]$ or $\genring = \witt[[u_1, \cdots, u_N]]/(p^m)$. Let $\deform{\generalhodge{\BT}}$ be a deformation of $\generalhodge{\BT}$ over $\genring$ with an isomorphism $\alpha: \deform{\generalhodge{\BT}} \otimes_\genring k \cong \generalhodge{\BT}$. Then there exists a unique filtration of $\deform{\BT}$
\[ 0 \subset \deform{\BT}_{r, r} \subset \deform{\BT}_{r-1, r} \subset \cdots \subset \deform{\BT}_{1, r} = \deform{\BT} \]
which lifts the slope filtration \eqref{slope filtration for hodge type} in the sense that $\alpha$ induces isomorphisms $\deform{\BT}_{s, r} \otimes_\genring k \cong \BT_{s,r}$ and $\generalhodge{\deform{\BT}_{s, r}/\deform{\BT}_{s+1, r}} \otimes_\genring k \cong \generalhodge{\BT}_s$ for $s = 1, 2, \cdots, r$. 
\end{prop}

\begin{proof} This is an immediate consequence of Theorem \ref{lifting of hodge newton filtration}. \end{proof}

\subsection{The canonical deformation of $\mu$-ordinary $p$-divisible groups}

\subsubsection{}
We now aim to find the canonical deformation $\deform{\generalhodge{\BT}}^\can$ of $\generalhodge{\BT}$ over $\witt$, which 
has the property that all endomorphisms of $\generalhodge{\BT}$ lifts to endomorphisms of $\deform{\generalhodge{\BT}}^\can$. When $G$ is of EL type, we already know existence of such a deformation thanks to the work of Moonen in \cite{Moonen04}. Our strategy is to deduce existence of $\deform{\generalhodge{\BT}}^\can$ from Moonen's result by means of an EL realization of the datum $(G, [b], \{\mu\})$. 

The following lemma is crucial for our strategy.

\begin{lemma}\label{EL realization of mu-ordinariness}
Let $(\EL{G}, [b], \{\mu\})$ be an EL realization of the datum $(G, [b], \{\mu\})$. Then $(\EL{G}, [b], \{\mu\})$ is $\mu$-ordinary. 
\end{lemma}
\begin{proof}
Consider the map on the Newton sets
\[\newtonset(G) \longrightarrow \newtonset(\EL{G})\]
induced by the embedding $G \hooklongrightarrow \EL{G}$. It maps $\bar{\mu}_G$ to $\bar{\mu}_{\EL{G}}$ by the proof of Lemma \ref{EL realization of hodge-newton type}, and $\newtonmap_G([b])$ to $\newtonmap_{\EL{G}}([b])$ by the functoriality of the Newton map. On the other hand, we have $\newtonmap_G([b]) = \bar{\mu}_G$ since $(G, [b], \{\mu\})$ is $\mu$-ordinary. Hence we deduce that $\newtonmap_{\EL{G}}([b]) = \bar{\mu}_{\EL{G}}$ which implies the assertion. 
\end{proof}

\subsubsection{} Let us now fix an EL realization $(\EL{G}, [b], \{\mu\})$ of the datum $(G, [b], \{\mu\})$. Then $(\EL{G}, [b], \{\mu\})$ is Hodge-Newton reducible with respect to some parabolic subgroup $\EL{\parabolic}$ of $\EL{G}$ with Levi factor $\EL{\levi}$ such that $\parabolic = \EL{\parabolic} \cap G$ and $\levi = \EL{\levi} \cap G$. In fact, since $\levi$ is the centralizer of $m \cdot \bar{\mu}$ in $G$, we may take $\EL{\parabolic}$ such that $\EL{\levi}$ is the centralizer of $m \cdot \bar{\mu}$ in $\EL{G}$. As in \ref{construction of sub local shimura data for hodge newton reducibility}, we assume for simplicity that $\EL{G}$ is of the form
\[\EL{G}:= \ELgroupintegral\] 
where $\integerring$ is the ring of integers of some finite unramified extension of $\Q_p$. Then $\EL{\levi}$ takes the form 
\begin{equation}\label{levi slope decomposition EL type}\EL{\levi} = \ELgroupintegral[j_1] \times \ELgroupintegral[j_2] \times \cdots \times \ELgroupintegral[j_r]. \end{equation}

We define $\EL{\levi}_j, \levi_j, b_j, \mu_j$ as in Theorem \ref{existence of hodge newton decomp in char p}. Then by the proof of Theorem \ref{existence of hodge newton decomp in char p} we have the following facts:
\begin{enumerate}
\item The tuples $(\levi_j, [b_j], \{\mu_j\})$ and $(\EL{\levi}_j, [b_j], \{\mu_j\})$ are unramified Shimura data of Hodge type,
\item Each factor $\generalhodge{\BT}_j$ in the slope decomposition \eqref{slope decomposition for hodge type} arises from the  datum $(\levi_j, [b_j], \{\mu_j\})$ with the choice $b_j \in [b_j]$. 
\end{enumerate}

Let $\EL{\BT}$ be the $p$-divisible group over $k$ with $\integerring$-module structure that arises from the datum $(\EL{G}, [b], \{\mu\})$ with the choice $b \in [b]$. It admits the Hodge-Newton decomposition
\begin{equation}\label{slope decomposition EL type} \EL{\BT} = \EL{\BT}_1 \times \EL{\BT}_2 \times \cdots \times \EL{\BT}_r\end{equation}
which gives rise to the slope decomposition \eqref{slope decomposition for hodge type} of $\generalhodge{\BT}$. By Lemma \ref{EL realization of mu-ordinariness}, the Newton polygon $\newtonmap_{\EL{G}}([b])$ and the $\sigma$-invariant Hodge polygon $\bar{\mu}_{\EL{G}}$ of $\EL{\BT}$ coincide. Since $\EL{\levi}$ is the centralizer of $m \cdot \bar{\mu}$ in $\EL{G}$, each factor in the decompositions \eqref{levi slope decomposition EL type} and \eqref{slope decomposition EL type} corresponds to a unique slope in the polygon $\bar{\mu}_{\EL{G}} = \newtonmap_{\EL{G}}([b])$. Hence the decomposition \eqref{slope decomposition EL type} is in fact the slope decomposition of $\EL{\BT}$.

\begin{prop}\label{rigidity for isoclinic objects hodge type} Each factor $\generalhodge{\BT}_j$ in the slope decomposition \eqref{slope decomposition for hodge type} is rigid, i.e., $\Def_{\BT_j, \levi_j}$ is pro-represented by $\witt$. 
\end{prop}
\begin{proof} 
Note that $\EL{\BT}_j$ arises from the datum $(\EL{\levi}_j, [b_j], \{\mu_j\})$ with the choice $b_j \in [b_j]$ (see the proof of Theorem \ref{existence of hodge newton decomp in char p}). It corresponds to a unique slope in the polygon $\bar{\mu}_{\EL{G}} = \newtonmap_{\EL{G}}([b])$, so it is $\mu$-ordinary with single slope. By \cite{Moonen04}, Corollary 2.1.5, its deformation space $\Def_{\BT_j, \EL{\levi}_j}$ is pro-represented by $\witt$. Now the assertion follows from the closed embedding of deformation spaces
\[ \Def_{\BT_j, \levi_j} \hooklongrightarrow \Def_{\BT_j, \EL{\levi}_j}\]
induced by the embedding $\levi_j \hooklongrightarrow \EL{\levi}_j$ (Lemma \ref{functoriality of deformation spaces}). 
\end{proof}

Let $\deform{\generalhodge{\BT}}_j^\can$ be the universal deformation of $\generalhodge{\BT}_j$ in the sense of Proposition \ref{universal deformation of a p-divisible group with tate tensors}. Proposition \ref{rigidity for isoclinic objects hodge type} says that $\deform{\generalhodge{\BT}}_j^\can$ is defined over $\witt$. Hence for any formally smooth $\witt$-algebra $\genring$ of the form $\genring = \witt[[u_1, \cdots, u_N]]$ or $\genring = \witt[[u_1, \cdots, u_N]]/(p^m)$, there exists a unique deformation of $\generalhodge{\BT}_j$ over $\genring$, namely $\deform{\generalhodge{\BT}}_j^\can \otimes_{\witt} \genring$. 

We define the \emph{canonical deformation} of $\generalhodge{\BT}$ to be a deformation of $\generalhodge{\BT}$ over $\witt$ given by
\[ \deform{\generalhodge{\BT}}^\can := \deform{\generalhodge{\BT}}_1^\can \times \deform{\generalhodge{\BT}}_2^\can \times \cdots \times \deform{\generalhodge{\BT}}_r^\can.\]
It is clear from this construction that all endomorphisms of $\generalhodge{\BT}$ lifts to endomorphisms of $\deform{\generalhodge{\BT}}^\can \otimes_{\witt} \genring$ for any formally smooth $\witt$-algebra $\genring$ of the form $\genring = \witt[[u_1, \cdots, u_N]]$ or $\genring = \witt[[u_1, \cdots, u_N]]/(p^m)$.





\subsection{Structure of deformation spaces}$ $

\subsubsection{} 

When $r=1$, we have $\Def_{\BT, G} \simeq \Spf (\witt)$ by Proposition \ref{rigidity for isoclinic objects hodge type}. 

Let us now consider the case $r=2$. Then we have the slope decompositions
\[ \generalhodge{\BT} = \generalhodge{\BT}_1 \times \generalhodge{\BT}_2 \quad \text{and} \quad \EL{\BT} = \EL{\BT}_1 \times \EL{\BT}_2.\]
Let $(d_s, \ELf_s)$ be the type of $\EL{\BT}_s$ for $s \in \{1, 2\}$ (see Example \ref{BT with O-module structure} for definition). Define a function $\ELf' : \ELgalois \to \{0, 1\}$ by
\[\ELf'(i) = \begin{cases} 0 \quad \text{if } \ELf_1(i) = \ELf_2(i) = 0;\\ 0 \quad \text{if } \ELf_1(i) = d_1 \text{ and } \ELf_2(i) = d_2; \\ 1 \quad \text{if } \ELf_1(i) = 0 \text{ and } \ELf_2(i) = d_2. \end{cases}\]
As noted in Example \ref{BT with O-module structure} for definition, there exists a unique isomorphism class of $\mu$-ordinary $p$-divisible group over $k$ with $\integerring$-module structure of type $(1, \ELf')$. We let $\EL{\deform{\BT}}^\can(1, \ELf')$ denote its canonical lifting.

\begin{theorem} \label{serre tate theorem for hodge type with two slopes} Notations above. The deformation space $\Def_{\BT, G}$ has a natural structure of a $p$-divisible group over $\witt$. More precisely, we have an isomorphism 
\[ \Def_{\BT, G} \cong \EL{\deform{\BT}}^\can(1, \ELf')^{d'}\]
as $p$-divisible groups over $\witt$ with $\integerring$-structure for some integer $d' \leq d_1 d_2$. 
\end{theorem}

\begin{proof}
Consider the category $\deformrings$ of artinian local $\witt$-algebra with residue field $k$. Let $\deform{\EL{\BT}}_j^\can$ denote the canonical deformation of $\EL{\BT}_j$ for $j=1, 2$. We define the functor 
\[\Ext(\deform{\EL{\BT}}_1^\can, \deform{\EL{\BT}}_2^\can): \deformrings \to \textbf{Sets}\] 
by setting $\Ext(\deform{\EL{\BT}}_1^\can, \deform{\EL{\BT}}_2^\can)(\genring)$ to be the set of isomorphism classes of extensions of $\deform{\EL{\BT}}_j^\can \otimes_{\witt} \genring$ by $\deform{\EL{\BT}}_2^\can \otimes_\witt \genring$ as fppf sheaves of $\integerring$-module. 

By \cite{Moonen04}, Theorem 2.3.3, we have the following isomorphisms:
\begin{enumerate}[label=(\alph*)]
\item\label{EL deform space isom to Ext} $\Def_{\BT, \EL{G}} \cong \Ext(\deform{\EL{\BT}}_1^\can, \deform{\EL{\BT}}_2^\can)$ as smooth formal groups over $\witt$,
\item\label{EL deform space isom to formal torus} $\Def_{\BT, \EL{G}} \cong \EL{\deform{\BT}}^\can(1, \ELf')^{d_1 d_2}$ as $p$-divisible groups over $\witt$ with $\integerring$-module structure. 
\end{enumerate}
On the other hand, by Lemma \ref{functoriality of deformation spaces} we have a closed embedding of deformation spaces
\begin{equation}\label{embedding of deformation space to EL realization} \Def_{\BT, G} \hooklongrightarrow \Def_{\BT, \EL{G}}.\end{equation}

Our first task is to show that $\Def_{\BT, G}$ is a subgroup of $\Def_{\BT, \EL{G}}$ with $\integerring$-module structure. Let $\genring$ be a smooth formal $\witt$-algebra of the form $\genring = \witt[[u_1, \cdots, u_N]]$ or $\genring = \witt[[u_1, \cdots, u_N]]/(p^m)$, and take two arbitrary deformations $\deform{\generalhodge{\BT}}$ and $\deform{\generalhodge{\BT}}'$ of $\generalhodge{\BT}$ over $\genring$. By Proposition \ref{lifting of the slope filtration for hodge type}, we have exact sequences
\begin{align*} 
0 \longrightarrow \deform{\generalhodge{\BT}}_1^\can  \otimes_\witt \genring \longrightarrow \deform{\generalhodge{\BT}} \longrightarrow \deform{\generalhodge{\BT}}_2^\can  \otimes_\witt \genring \longrightarrow 0, \\
0 \longrightarrow \deform{\generalhodge{\BT}}_1^\can  \otimes_\witt \genring \longrightarrow \deform{\generalhodge{\BT}}' \longrightarrow \deform{\generalhodge{\BT}}_2^\can  \otimes_\witt \genring \longrightarrow 0.
\end{align*}
We denote by $\deform{\BT} \odot \deform{\BT}'$ the underlying $p$-divisible group of their Baer sum taken in $\Ext(\deform{\EL{\BT}}_1^\can, \deform{\EL{\BT}}_2^\can)(\genring)$.

We wish to show that $\deform{\BT} \odot \deform{\BT}' \in \Def_{\BT, G}(\genring)$. By the isomorphism \ref{EL deform space isom to Ext}, we already know that $\deform{\BT} \odot \deform{\BT}' \in \Def_{\BT, \EL{G}}(\genring)$. Hence it remains to show that we have tensors on (the Dieudonn\'e module of) $\deform{\BT} \odot \deform{\BT}'$ which lift the tensors $(t_i)$ on $\generalhodge{\BT}$ in the sense of Proposition \ref{universal deformation of a p-divisible group}. Unfortunately, it is not easy to explicitly find these tensors in terms of the tensors on $\deform{\generalhodge{\BT}}$ and $\deform{\generalhodge{\BT}}'$. Instead, we start with the family of all tensors $(\mathfrak{s}_j)$ on $\module$ which are fixed by $G$. Then we have a family $(\mathfrak{t}_j):= (\mathfrak{s}_j \otimes 1)$ on $\module \otimes_{Z_p} \witt \simeq \dieudonnemodule$. Since the formal deformation space $\Def_{\BT, G}$ is independent of the choice of tensors $(t_i)$, we get tensors $(\bar{\mathfrak{t}}_j)$ on $\deform{\generalhodge{\BT}}$ and $(\bar{\mathfrak{t}}'_j)$ on $\deform{\generalhodge{\BT}}'$ which lift $(\mathfrak{t}_j)$ (in the sense of Proposition \ref{universal deformation of a p-divisible group}). Moreover, the families $(\bar{\mathfrak{t}}_j)$ and $(\bar{\mathfrak{t}}'_j)$ map to the same family of tensors on $\deform{\generalhodge{\BT}}_2^\can$ under the surjections $\deform{\generalhodge{\BT}} \twoheadrightarrow \deform{\generalhodge{\BT}}_2^\can$ and $\deform{\generalhodge{\BT}}'\twoheadrightarrow \deform{\generalhodge{\BT}}_2^\can$. Hence the families $(\bar{\mathfrak{t}}_j)$ and $(\bar{\mathfrak{t}}'_j)$ define the same family of tensors on $\deform{\BT} \odot \deform{\BT}'$ which lift $(\mathfrak{t}_j)$. In particular, there exists a family of tensors on $\deform{\BT} \odot \deform{\BT}'$ which lift $(t_i)$.

Since $\Def_{\BT, \EL{G}}$ has a finite $p$-torsion for being a $p$-divisible group, we observe from the embedding \eqref{embedding of deformation space to EL realization} that $\Def_{\BT, G}$ also has finite finite $p$-torsion. Using the same argument as in the proof of \cite{Moonen04}, Theorem 2.3.3, we deduce that $\Def_{\BT, G}$ is a $p$-divisible group.

Hence $\Def_{\generalhodge{\BT}}$ is a $p$-divisible subgroup of $\Def_{\BT, \EL{G}} \cong \EL{\deform{\BT}}^\can(1, \ELf')^{d_1 d_2}$ with $\integerring$-module structure. 
Now the dimension of $\Def_{\BT, G}$ determines an integer $d'$ such that 
\[ \Def_{\BT, G} \cong \EL{\deform{\BT}}^\can(1, \ELf')^{d'}\]
as $p$-divisible groups over $\witt$ with $\integerring$-module structure. \end{proof}

\begin{remark}
From the proof, one sees that the canonical deformation $\deform{\generalhodge{\BT}}^\can$ corresponds to the identity element in the $p$-divisible group structure of $\Def_{\BT, G}$. 
\end{remark}

\subsubsection{}
We finally consider the case $r \geq 3$. For convenience, we write $\Def_{\EL{\BT}_{a,b}}$ for the deformation space of $\EL{\BT}_{a, b}$. These spaces fit into a diagram
\begin{center}
\begin{tikzpicture}[description/.style={fill=white,inner sep=2pt}]
\matrix (m) [matrix of math nodes, row sep=1.5em,
column sep=0.2em, text height=1.5ex, text depth=0.25ex]
{ & & & \Def_{\EL{\BT}_{1, r}}=\Def_{\BT, \EL{G}} & & & \\
& &\Def_{\EL{\BT}_{1, r-1}} & & \Def_{\EL{\BT}_{2, r}} & & \\
& \Def_{\EL{\BT}_{1, r-2}} & &\Def_{\EL{\BT}_{2, r-1}} & & \Def_{\EL{\BT}_{3, r}}& \\
\cdots & & \cdots & & \cdots & & \cdots \\
};
\path[-stealth]
    (m-1-4) edge (m-2-3)
    (m-1-4) edge (m-2-5)
    (m-2-3) edge (m-3-2)
    (m-2-3) edge (m-3-4)
    (m-2-5) edge (m-3-4)
    (m-2-5) edge (m-3-6)
    (m-3-2) edge (m-4-1.north)
    (m-3-2) edge (m-4-3)
    (m-3-4) edge (m-4-3)
    (m-3-4) edge (m-4-5)
    (m-3-6) edge (m-4-5)
    (m-3-6) edge (m-4-7.north);
\end{tikzpicture}\end{center}
where each map comes from the restriction of the filtration in Proposition \ref{lifting of the slope filtration for hodge type} (see \cite{Moonen04}, 2.3.6.). This diagram carries some additional structures called the \emph{cascade structure}, as described by Moonen in loc. cit.

We denote by $\Def_{\generalhodge{\BT}_{a,b}}$ the pull back of $\Def_{\EL{\BT}_{a,b}}$ over $\Def_{\BT, G}$. Then $\Def_{\generalhodge{\BT}_{a,b}}$ classifies deformations of $\BT_{a, b}$ with a filtration that comes from the filtration of $\deform{\BT}$ in Proposition \ref{lifting of the slope filtration for hodge type}. If we pull back the above diagram over over $\Def_{\BT, G}$, we get another diagram
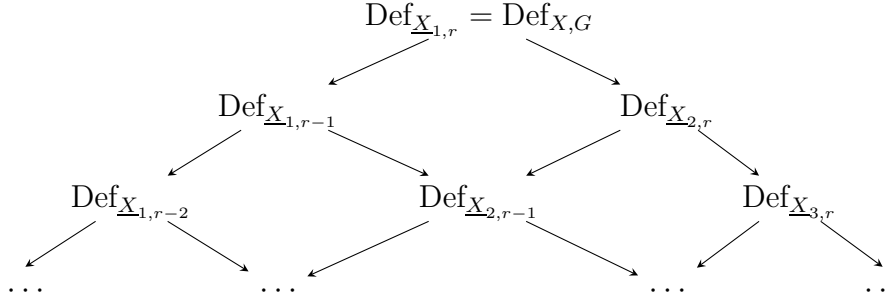
\begin{center}
\begin{tikzpicture}[description/.style={fill=white,inner sep=2pt}]
\matrix (m) [matrix of math nodes, row sep=1.5em,
column sep=0.2em, text height=1.5ex, text depth=0.25ex]
{ & & & \Def_{\generalhodge{\BT}_{1, r}}=\Def_{\BT, G} & & & \\
& &\Def_{\generalhodge{\BT}_{1, r-1}} & & \Def_{\generalhodge{\BT}_{2, r}} & & \\
& \Def_{\generalhodge{\BT}_{1, r-2}} & &\Def_{\generalhodge{\BT}_{2, r-1}} & & \Def_{\generalhodge{\BT}_{3, r}}& \\
\cdots & & \cdots & & \cdots & & \cdots \\
};
\path[-stealth]
    (m-1-4) edge (m-2-3)
    (m-1-4) edge (m-2-5)
    (m-2-3) edge (m-3-2)
    (m-2-3) edge (m-3-4)
    (m-2-5) edge (m-3-4)
    (m-2-5) edge (m-3-6)
    (m-3-2) edge (m-4-1.north)
    (m-3-2) edge (m-4-3)
    (m-3-4) edge (m-4-3)
    (m-3-4) edge (m-4-5)
    (m-3-6) edge (m-4-5)
    (m-3-6) edge (m-4-7.north);
\end{tikzpicture}\end{center}
where each map comes from the restriction of the filtration in Proposition \ref{lifting of the slope filtration for hodge type}. With similar arguments as in the proof of Theorem \ref{serre tate theorem for hodge type with two slopes}, one can give a group structure on $\Def_{\generalhodge{\BT}_{a, b}}$ over $\Def_{\generalhodge{\BT}_{a, b-1}}$ and $\Def_{\generalhodge{\BT}_{a+1, b}}$ (cf. \cite{Moonen04}, 2.3.6.).  However, this diagram does not carry the full cascade structure in general.

\section{Congruence relations on Shimura varieties of Hodge type}

In this section, we use our generalization of Serre-Tate deformation theory developed in \S\ref{serre tate theory for hodge type} to study some congruence relations on Shimura varieties of Hodge type. Our proof will closely follow Moonen's proof for PEL case in \cite{Moonen04}, \S4. 

\subsection{Stratification on the special fiber}

\subsubsection{} \label{notations for shimura varieties}

Let us first set up some notations for Shimura varieties of Hodge type. 

Let $\shimuradatum$ be a Shimura datum of Hodge type. This means that it admits an embedding into a symplectic Shimura datum
\[\shimuradatum \hookrightarrow (\text{GSp}, S^{\pm}).\]  
We fix such an embedding for the rest of this section. 

For each $h \in \mathfrak{H}$, we define the cocharacter $\lambda_h$ of $\charzero{G}_\C$ by 
\[\lambda_h: \C^\times \longrightarrow \C^\times \times c^*(\C^\times) \cong \text{Res}_{\C/\R}(\C) \stackrel{h}{\longrightarrow} G(\C)\]
where $c$ denotes the complex conjugation. We denote by $\{\lambda\inv\}$ the unique $G(\C)$-conjugacy class which contains all $\lambda_h\inv$. Let $\reflexfield$ be the field of definition of $\{\lambda\inv\}$, called the \emph{reflex field} of $\shimuradatum$. 
We write $\reflexring$ for the ring of integers in $\reflexfield$.

We assume that $\charzero{G}$ is connected and of good reduction at $p$. Then $\charzero{G}_{\Q_p}$ is unramified, so we can take a reductive model $G$ of $\charzero{G}$ over $\Z_p$. Moreover, we can choose a $\Z_p$-lattice $\module$ and an embedding $G \hookrightarrow \GL(\module)$ which induces the embedding $\charzero{G}_{\Q_p} \hookrightarrow \text{GSp}_{\Q_p}$ (see \cite{Kisin10}, 2.3.2.). We choose a finite family of tensors $(s_i)$ on $\module$ whose pointwise stabilizer is $G$. We also fix a Borel pair $(\borel, T)$ of $G$ and take $\mu\in \cocharacters(T)$ to be a unique dominant cocharacter such that $\sigma\inv(\mu) \in \{\lambda\inv\}$. 

Take $\levelatp := G(\Z_p)$. For a sufficiently small open and compact subgroup $\levelawayp$ of $\charzero{G}(\Afp)$, the double quotient
\[ \Sh_{\levelatp \levelawayp} \shimuradatum:= \charzero{G}(\Q) \backslash \mathfrak{H} \times \charzero{G}(\Af) / \levelatp \levelawayp\]
has a natural structure as a smooth quasi-projective variety over $\C$. Moreover, it has a canonical model over $\reflexfield$, which we also denote by $\Sh_{\levelatp \levelawayp} \shimuradatum$. Then we can define the pro-variety 
\[ \Sh_\levelatp \shimuradatum := \varprojlim_\levelawayp \Sh_{\levelatp \levelawayp} \shimuradatum\]
where the limit is taken over the set of open and compact subgroups of $\charzero{G}(\Afp)$. This is a scheme over $\reflexfield$ with a continuous right action of $\charzero{G}(\Afp)$ as described in \cite{Deligne79}, 2.7.1. or \cite{Milne92}, 2.1.

\subsubsection{} Fix a place $\place$ of $\reflexfield$ over $p$, and let $\localring$ be the localization of $\reflexring$ at $\place$. In \cite{Kisin10}, Kisin constructed an integral canonical model of $\Sh_\levelatp \shimuradatum$
\[ \integralmodel_\levelatp \shimuradatum = \varprojlim_\levelawayp \integralmodel_{\levelatp \levelawayp} \shimuradatum.\]
By definition, this is a scheme over $\localring$ with a continuous right action of $\charzero{G}(\Afp)$ satisfying the following properties:

\begin{enumerate}[label=(\roman*)]
\item $\integralmodel_\levelatp \shimuradatum \otimes_{\localring} \reflexfield$ is $\charzero{G}(\Afp)$-equivariantly isomorphic to $\Sh_\levelatp \shimuradatum$,

\item for sufficiently small $\levelawayp$, $\integralmodel_{\levelatp \levelawayp} \shimuradatum$ is smooth over $\localring$ and the connecting morphisms in $\varprojlim_\levelawayp \integralmodel_{\levelatp \levelawayp} \shimuradatum$ are \'etale,

\item if $Y$ is a regular, formally smooth $\localring$-scheme, every morphism $Y \otimes_{\localring} \reflexfield \to \integralmodel_\levelatp \shimuradatum \otimes_{\localring} \reflexfield$ extends to a morphism $Y \to \integralmodel_\levelatp \shimuradatum$. 
\end{enumerate}
We fix a model $\integralmodel := \integralmodel_\level \shimuradatum$ associated to some sufficiently small subgroup $\levelawayp \subseteq \charzero{G}(\Afp)$. By construction, it comes with a universal abelian scheme $\univabelsch \rightarrow \integralmodel$. For a point $x$ on $\integralmodel$, let $\univabelsch_x$ denote the corresponding abelian scheme, and take $b$ to be the linearization of the Frobenius map on $\D(\univabelsch_x[p])$. Then the tuple $(G, [b], \{\mu\})$ is an unramified local Shimura datum of Hodge type which gives rise to $G$-structure on the $p$-divisible group $\univabelsch_x[p]$.

Let $\reflexresiduefield$ be the residue field of $\localring$. We define the \emph{$\mu$-ordinary locus} to be the set
\[ \integralmodel^\ord := \{ x \in \integralmodel \otimes \reflexresiduefield : \univabelsch_x[p] \text{ is }\mu\text{-ordinary} \}.\]
By the work of Wortmann in \cite{Wortmann13}, we know  that the $\mu$-ordinary locus is open and dense in $\integralmodel \otimes \reflexresiduefield$.

\subsection{Congruence relations}$ $


\subsubsection{} 

Consider the product $\integralmodel \times \integralmodel$ of $\integralmodel$ with itself over $\localring$. We obtain two abelian schemes $\univabelsch_1, \univabelsch_2 \to \integralmodel \times \integralmodel$ by pulling back the universal abelian scheme $\univabelsch \to \integralmodel$ via the two projections. Then we have a relative scheme
\[ \pisog \to \integralmodel \times \integralmodel\]
which classifies the $p$-isogenies between $\univabelsch_1$ and $\univabelsch_2$ that preserves $G$-structure on the $p$-divisible groups. Define $\pisog^\ord$ to be the inverse image of $\integralmodel^\ord \times \integralmodel^\ord$. We write $\pisoggenfiber$ (resp. $\pisoggenfiber^\ord$) and $\pisogspecialfiber$ (resp. $\pisogspecialfiber^\ord$) for the generic fiber and the special fiber of $\pisog$ (resp. $\pisog^\ord$).

\subsubsection{}

Let $\localring \to \genfield$ be a homomorphism with $\genfield$ a field. If char$(\genfield)=0$, we define $\Q[\pisoggenfiber \otimes \genfield]$ to be the $\Q$-space freely generated by the irreducible components of $\pisoggenfiber \otimes \genfield$. Similarly, if char$(\genfield)=p$, we define $\Q[\pisogspecialfiber^\ord \otimes \genfield]$ to be the $\Q$-space freely generated by the irreducible components of $\pisogspecialfiber^\ord \otimes \genfield$. 

Let us define a $\Q$-algebra structure on these $\Q$-spaces. The two projections of $\integralmodel \times \integralmodel$ gives two morphisms 
\[s, t: \pisog \to \integralmodel,\] 
sending a $p$-isogeny to its source and target, respectively. In addition, the composition of isogenies defines a morphism 
\[ c: \pisog \times_{t, s} \pisog \to \pisog. \]
One can show that these morphisms are proper using the valuative criterion. For two cycles $Y_1, Y_2$ on $\pisog \otimes \genfield$, we define
\[ Y_1 \cdot Y_2 := c_*(Y_1 \times_{t, s} Y_2).\]
This product defines a desired $\Q$-algebra structure on $\Q[\pisoggenfiber \otimes \genfield]$ and $\Q[\pisogspecialfiber^\ord \otimes \genfield]$, as we have the following lemma:
\begin{lemma}
If $Y_1$ and $Y_2$ are irreducible components of $\pisog \otimes \genfield$, then $Y_1 \cdot Y_2$ is a $\Q$-linear combination of irreducible components. 
\end{lemma}
\begin{proof}
The proof is essentially identical to the proof for the Siegel modular case or the PEL case. The main point is that the morphisms $s$ and $t$ are finite and flat over $\genfield$ if $\text{char}(\genfield) = p$. See \cite{Moonen04}, Lemma 4.2.2. 
\end{proof}

\subsubsection{} 

Let $q=p^m$ be the cardinality of the residue field $\reflexresiduefield$. We have a section $\frobcorrsection: \integralmodel \otimes \reflexresiduefield \to \pisogspecialfiber$ of the source morphism, sending a point $x \in  \integralmodel \otimes \reflexresiduefield$ to the $m$-th power Frobenius isogeny on $\univabelsch_x$. Let $\frobcorr$ denote its image, which is a closed reduced subscheme of $\pisogspecialfiber$. In fact, it is a union of irreducible components of $\pisogspecialfiber$, as the source morphism $s$ is finite and flat. This allows us to consider $\frobcorr$ as an element of $\Q[\pisogspecialfiber]$, or as an element of $\Q[\pisogspecialfiber^\ord]$. We refer to this element as the Frobenius correspondence.

\subsubsection{}

Let $\heckealg{G}$ be the Hecke algebra of $G$ with respect to its hyperspecial subgroup $G(\Z_p)$. Define $\heckealgsupp{G} \subset \heckealg{G}$ to be the subalgebra of $\Q$-valued functions that have support contained in $G(\Q_p) \cap \End(\module)$. For the centralizer $\levi$ of $\bar{\mu}$ in $G$, we can similarly the Hecke algebras $\heckealgsupp{\levi} \subset \heckealg{\levi}$ (see \cite{Wedhorn00}, \S1). Then we have a homomorphism 
\[ \twistedsatake: \heckealg{G} \to \heckealg{\levi},\]
called the twisted Satake homomorphism. It restricts to a map $\heckealgsupp{G} \to \heckealgsupp{\levi}$, which we denote by the same symbol.

\subsubsection{} \label{type of isogeny}

Take $\genfield$ to be a field containing $\reflexfield$, and let $f: \univabelsch_{x_1} \to \univabelsch_{x_2}$ be an isogeny corresponding to an $\genfield$-valued point of $\pisoggenfiber$. Write $\BT^{(i)} := \univabelsch_{x_i}[p]$ for $i \in \{1, 2\}$. The identification $\integralmodel \otimes_{\localring} \reflexfield \cong \Sh_{\level}\shimuradatum$ gives us identifications of Tate-modules $\alpha_i : \module \stackrel{\sim}{\to} \tatemodule{\BT^{(i)}}$ for $i \in \{1, 2\}$, which are canonical up to the action of an element of $G(\Z_p)$. We also have an induced linear isomorphism $\rattatemodule{f}: \rattatemodule{\BT^{(1)}} \stackrel{\sim}{\to} \rattatemodule{\BT^{(2)}}$ on the rational Tate modules. Then the map 
\[\alpha_2\inv \circ \rattatemodule{f} \circ \alpha_1: \module \otimes \Q_p \stackrel{\sim}{\to} \module \otimes \Q_p\]
is an element of $G(\Q_p)$, and its class in $G(\Z_p) \backslash G(\Q_p) / G(\Z_p)$ is independent of the choice of the $\alpha_i$. We refer to this class as the \emph{type} of the $p$-isogeny $f$.

The type of an isogeny is constant on irreducible components of $\pisoggenfiber$. To every double coset $G(\Z_p) \gamma G(\Z_p)$ with $\gamma \in G(Q_p) \cap \End(\module)$, we associate the sum of all irreducible components of $\pisoggenfiber \otimes \genfield$ where the $p$-isogeny has type $G(\Z_p) \gamma G(\Z_p)$. This defines a map
\[ h: \heckealgsupp{G} \to \Q[\pisoggenfiber \otimes \genfield ].\]
which is a $\Q$-algebra homomorphism.

\subsubsection{} \label{p-type of isogeny}

Let us now take $\genfield$ to be a perfect field containing $\reflexresiduefield$. Let $x$ be a point in $\integralmodel^\ord$ and write $\generalhodge{\BT}$ for the $p$-divisible group $\univabelsch_x[p]$ with $G$-structure. Since $\generalhodge{\BT}$ is $\mu$-ordinariy, it admits a slope decomposition 
\[\generalhodge{\BT} = \generalhodge{\BT}_1 \times \generalhodge{\BT}_2 \times \cdots \times \generalhodge{\BT}_r \]
and the canonical deformation 
\[ \candeform{\BT} = \deform{\generalhodge{\BT}}^\can_1 \times \deform{\generalhodge{\BT}}^\can_2 \times \cdots \times \deform{\generalhodge{\BT}}^\can_r.\]
Then we have a decomposition
\begin{equation}\label{slope decomposition for tate module} 
\tatemodule{\candeform{\BT}} = \tatemodule{\deform{\generalhodge{\BT}}_1} \oplus \tatemodule{\deform{\generalhodge{\BT}}_2} \oplus \cdots \oplus \tatemodule{\deform{\generalhodge{\BT}}_r}.
\end{equation}
On the other hand, we have an identification $\alpha: \module \stackrel{\sim}{\to} \tatemodule{\generalhodge{\deform{\BT}}^\can}$ as in \ref{type of isogeny}. As in the PEL case, one can prove that, after changing $\alpha$ by an element of $G(\Z_p)$, the decomposition \eqref{slope decomposition for tate module} agrees with the eigenspace decomposition of $\module$ with respect to $\bar{\mu}$ (see \cite{Moonen04}, Lemma 4.2.9.). 

Let $f: \univabelsch_{x_1} \to \univabelsch_{x_2}$ be an isogeny corresponding to an $\genfield$-valued point of $\pisoggenfiber$, and write $\BT^{(i)} := \univabelsch_{x_i}[p]$ for $i \in \{1, 2\}$. Choose identifications $\alpha_i : \module \stackrel{\sim}{\to} \tatemodule{\BT^{(i)}}$ for $i \in \{1, 2\}$ as above, and let $\rattatemodule{f}: \rattatemodule{\BT^{(1)}} \stackrel{\sim}{\to} \rattatemodule{\BT^{(2)}}$ be the linear isomorphism induced by $f$. Then the map
\[\alpha_2\inv \circ \rattatemodule{f} \circ \alpha_1: \module \otimes \Q_p \stackrel{\sim}{\to} \module \otimes \Q_p\]
is an element of $\levi(\Q_p)$. We define the \emph{$p$-type} of $f$ to be the class of this map in $\levi(\Z_p) \backslash \levi(\Q_p) / \levi(\Z_p)$, which is independent of the choice of the $\alpha_i$.

The same argument as in \cite{Moonen04}, Lemma 4.2.11. shows that the $p$-type of an isogeny is locally constant on $\pisogspecialfiber$. As in \ref{type of isogeny}, this allows us to define a map 
\[ \bar{h}: \heckealgsupp{\levi} \to \Q[\pisoggenfiber \otimes \genfield ].\]

\begin{theorem} Let $\sigma : \Q[J] \to \Q[\charzero{J}_0^\ord]$ be the homomorphism given by specialization of cycles. Then we have a commutative diagram of $\Q$-algebra homomorphisms
\begin{center}
\begin{tikzpicture}[description/.style={fill=white,inner sep=2pt}]
\matrix (m) [matrix of math nodes, row sep=5em,
column sep=3em, text height=1.5ex, text depth=0.25ex]
{ \heckealgsupp{G} & \Q[\pisoggenfiber] \\
\heckealgsupp{\levi} & \Q[\pisogspecialfiber^\ord] \\
};
\path[-stealth]
    (m-1-1.east) edge node[above]{$h$} (m-1-2)
    (m-1-1.south) edge node[left]{$\twistedsatake$} (m-2-1)
    (m-2-1.east) edge node[above]{$\bar{h}$} (m-2-2)
    (m-1-2.south) edge node[right]{$\sigma$} (m-2-2);
\end{tikzpicture}
\end{center}
\end{theorem}

\begin{proof}
The proof is essentially identical to the proof for the Siegel modular case. See \cite{Chai-Faltings90}, p. 263 or \cite{Moonen04}, Theorem 4.2.13.
\end{proof}


\begin{cor} Let $\Phi$ be the Frobenius correspondence on $\integralmodel_0$. Let $H_{\shimuradatum} \in H_0(\charzero{G}, \Q)[t]$ be the Hecke polynomial associated to the Shimura datum $\shimuradatum$, as defined in \cite{Wedhorn00}, \S 2. Regarding $\Q[\charzero{J}^\ord_0]$ as an algebra over $H_0(\charzero{G}, \Q)$ via $\sigma \circ h$, we have the relation $H_{\shimuradatum}(\Phi) = 0$. 
\end{cor}

\begin{proof}
This is a direct consequence of the theorem together with some purely group theoretic results due to B\"ultel. See \cite{Moonen04}, Corollary 4.2.14. 
\end{proof}

\begin{cor} If $\charzero{J}_0^\ord$ is Zariski dense in $\charzero{J}_0$ then the relation 
$H_{\shimuradatum}(\Phi) = 0$ holds in the algebra $\Q[\charzero{J}_0]$, viewed as an algebra over $H_0(\charzero{G}, \Q)$ via $\sigma \circ h$. 
\end{cor}

\bigskip



\bibliographystyle{chicagonr}
\bibliography{strings,kcb}

 \end{document}